\documentclass[12pt,epsfig,amsfonts]{amsart} 
\usepackage{amsmath,amsthm,amssymb,amscd,epsfig,color}
\usepackage{graphicx}
\usepackage{mathrsfs}
\usepackage{ulem}

\newtheorem{prop}{Proposition}[section]
\newtheorem{lemma}[prop]{Lemma}
\newtheorem{thm}[prop]{Theorem}
\newtheorem{cor}[prop]{Corollary}

\theoremstyle{definition}

\newcommand{\1}{\mbox{1}\hspace{-0.25em}\mbox{l}}
\usepackage{url}

\newtheorem{remark}[prop]{Remark}

\numberwithin{equation}{section}

\usepackage[utf8]{inputenc}
\usepackage{fancyvrb}

\newcommand{\confrac}[2]{%
  \frac{\displaystyle{%
    \strut\hfill{#1}\hfill\;\vrule}}%
      {\displaystyle{%
       \strut\vrule\;\hfill{#2}\hfill}}}%

\begin{document}

\author{Shintaro Suzuki and Hiroki Takahasi}

\address{Department of Mathematics,
Tokyo Gakugei University, 4-1-1 Nukuikita-machi Koganei-shi, Tokyo,  184-8501, JAPAN}
\email{shin05@u-gakugei.ac.jp}
\address{ Keio Institute of Pure and Applied Sciences (KiPAS), Department of Mathematics,
Keio University, Yokohama,
223-8522, JAPAN} 
\email{hiroki@math.keio.ac.jp}


\title[ 
 Representations of the Gauss-R\'enyi measure by ``periodic points'']
{Annealed and quenched representations of the Gauss-R\'enyi measure by ``periodic points''} 

\begin{abstract}
We consider independently identically distributed random compositions 
of the Gauss and R\'enyi maps that 
are related to Diophantine approximations.
Elaborating on methods in ergodic theory, thermodynamic formalism and large deviations, 
we show that weighted cycles of this random dynamical system 
equidistribute with respect to the Gauss-R\'enyi measure.
We present both annealed (sample-averaged) and quenched (samplewise) results.
\end{abstract}
\maketitle

\tableofcontents

\section{Introduction}

One leading idea in the qualitative theory of deterministic dynamical systems is 
to use the collection of periodic orbits as a spine 
to structure the dynamics. This idea traces back to 
Poincar\'e 
\cite{Po}: 
``{\it ... ce qui nous rend ces solutions p\'eriodiques si pr\'ecieuses, ... 
la seul br\`eche par o\`u nous puissions esseyer de p\'en\'etrer dans une place jusqu'ici r\'eput\'ee inabordable.}'' 
Bowen's pioneering results \cite{Bow71,Bow74} 
assert that periodic points of topologically mixing Axiom~A diffeomorphisms equidistribute with respect to the measure of maximal entropy. The importance of periodic orbits in descriptions of ergodic properties of natural invariant probability measures has long been recognized in the physics literature, see e.g., \cite{Cv91,GOY88}.
Cvitanovi\'c \cite{Cv91} proposed 
expansions
of dynamical characteristics into series or products that consist of infinitely many periodic orbits, to better analyze the characteristics taking advantage of the simple structure of each periodic orbit in the expansions.

By deterministic dynamical systems, we mean ordinary differential equations or iterated maps.
Systems with multiple evolution laws, called {\it random dynamical systems} \cite{Ar98}, are also relevant to consider. 
For a large class of random dynamical systems, 
 we expect that 
 periodic orbits
still play significant roles, but
it is not clear 
how periodic points 
should be defined. 

In discrete time, deterministic dynamical systems are iterations of one fixed map, whereas random dynamical systems are compositions of different maps chosen at random.
A naive idea is to use fixed points of random compositions of $n$ maps as substitutes for periodic points of period $n$. Such ``periodic points'' have been indeed considered, 
see e.g., \cite{Buz02,Rue90,SuzTak21}. 
For other substitutes for the concept of periodic points in the context random dynamical systems, see e.g.,  \cite{FJJK05,JK16,Kif00}.

 In  \cite{SuzTak21}, the authors 
   proved an analogue of Bowen's equidistribution theorem \cite{Bow71,Bow74}  for 
   random dynamical systems generated by 
   a class of 
   interval maps
   with finitely many branches.
   The aim of this paper is to extend this analogue
 to random dynamical systems generated by 
 the Gauss and R\'enyi maps.
 The Gauss map
 $T_0\colon(0,1]\to[0,1)$ and the R\'enyi map $T_{1}\colon[0,1)\to[0,1)$ are
respectively given by
 \[T_{0}x=\frac{1}{x}-\left\lfloor \frac{1}{x}\right\rfloor\quad\text{and}\quad 
 T_1x=\frac{1}{1-x}-\left\lfloor \frac{1}{1-x}\right\rfloor.\]
          The graph of $T_1$ is obtained by reversing the graph of  $T_0$ around the axis $\{x=1/2\}$, as shown in \textsc{Figure}~\ref{fig1}. 
Since both maps have infinitely many branches, the random dynamical systems they generate are beyond the scope of  
 \cite{SuzTak21}.

\begin{figure}
\begin{center}
\includegraphics[height=5cm,width=6.875cm]{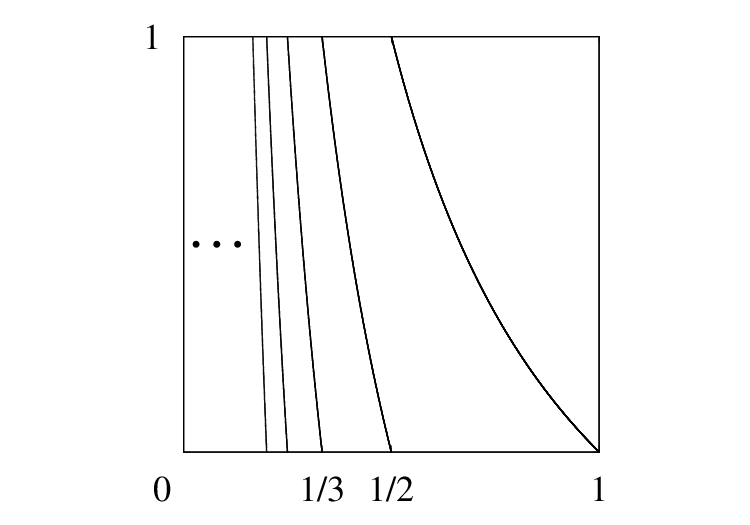}
\includegraphics[height=5cm,width=6.875cm]{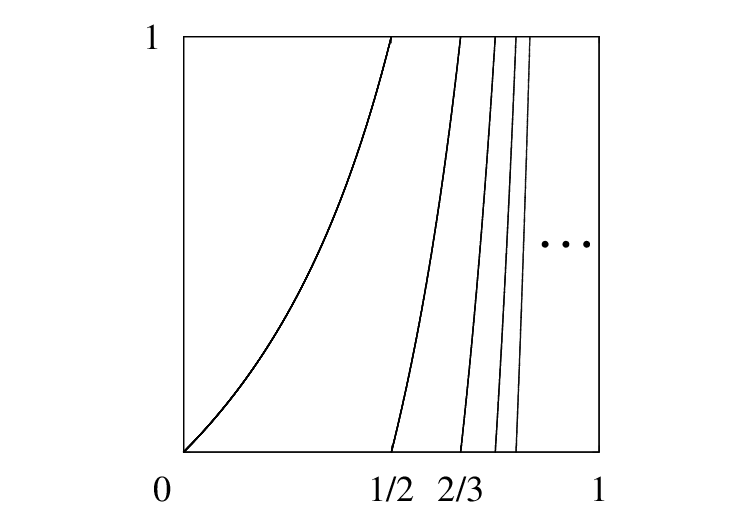}
\caption
{The graph of the Gauss map $T_0$ (left) and that of the R\'enyi map $T_1$ (right):
$T_0^{-1}(0)=\{1/k\colon k\in\mathbb N\}$, $T_1^{-1}(0)=\{(k-1)/k\colon k\in\mathbb N\}$; $T_0^{-1}(1)=T_1^{-1}(1)=\emptyset$;
$T_10=0$, $T_1'0=1$.
}\label{fig1}

\end{center}
\end{figure}

 For a sample path $\omega=(\omega_n)_{n=1}^\infty$
  in the product space
    $\Omega=\{0,1\}^{\mathbb N}$ of the discrete space $\{0,1\}$, 
  we consider a random composition 
  \[T_\omega^n=T_{\omega_n}\circ T_{\omega_{n-1}}\circ \cdots\circ T_{\omega_1}\ \text{ for } n\in\mathbb N.\] 
  Write
$T_\omega^0$ for the identity map on $[0,1]$.
 Let
$\Lambda_\omega$ denote the set of $x\in[0,1]$ such that $T_\omega^nx$ is defined for every $n\in\mathbb N$. Each $x\in\Lambda_\omega$
has a continued fraction expansion 
\begin{equation}\label{r-expansion}
x=
\omega_1+\confrac{(-1)^{\omega_1} }{C_{1}(\omega,x)} + \confrac{(-1)^{\omega_2} }{C_{2}(\omega,x)}  + \confrac{(-1)^{\omega_3} }{C_{3}(\omega,x)}+\cdots,
\end{equation}
where each
 $C_n(\omega,x)$, $n\in\mathbb N$ is a positive integer that is determined by
  $T_\omega^{n-1}x$, $\omega_{n}$, $\omega_{n+1}$, and satisfies
  $(-1)^{\omega_{n+1}}+C_{n}(\omega,x)\geq1$
  (see $\S$\ref{random-s} for details).
  This type of continued fractions was first considered by Perron \cite{Per50}.
In the case $\omega_n=0$ for all $n\in\mathbb N$ we obtain the well-known {\it regular continued fraction}
\[
x=\confrac{1 }{A_{1}(x)} + \confrac{1 }{A_{2}(x)}  + \confrac{1 }{A_{3}(x)}+\cdots,\]
where $A_n(x)=\lfloor 1/T_0^{n-1}x\rfloor$ for  $n\in\mathbb N$. In the case $\omega_n=1$ for all $n\in\mathbb N$ we obtain the {\it backward continued fraction} \[
x=
1-\confrac{1 }{B_{1}(x)} - \confrac{1 }{B_{2}(x)} -\confrac{1 }{B_{3}(x)} -\cdots,\]
where $B_n(x)=\lfloor 1/(1-T_1^{n-1}x)\rfloor+1$ for $n\in\mathbb N$.
 The backward continued fraction was used, for example, in computing certain inhomogeneous approximation constants    \cite{Pin01}. 
    For its connection with geodesic flows, see \cite{AdlFla84}.

It is the essential difference between
statistical properties of the  sequences $(A_n(x))_{n=1}^\infty$ and  $(B_n(x))_{n=1}^\infty$ that makes the random continued fraction interesting.
For Lebesgue almost every irrational $x$ in $(0,1)$,
each positive integer $k$ appears in  $(A_n(x))_{n=1}^\infty$ with frequency $\frac{1}{\log2}\log\frac{(k+1)^2}{k(k+2)}$, 
while  the frequency of $2$ in $(B_n(x))_{n=1}^\infty$ is $1$.
This is due to the fact that 
$T_0$ leaves invariant {\it the Gauss measure}
 $d\lambda_0=\frac{1}{\log2}\frac{dx }{x+1}$, 
while $T_1$ leaves invariant the infinite measure $\frac{dx}{x}$. More precisely, $x=0$ is a neutral fixed point of $T_1$: $T_10=0$ and $T_1'0=1$.
For more comparisons of the regular and backward continued fractions as well as more information on the singular behavior of the digit sequence in the backward continued fraction, see \cite{Aar86,AarNak03,Io10,IosKra02,PolWei99,Tak19,Tak22} for example.  

\subsection{Statements of results}
We consider an independently identically distributed (i.i.d.) random dynamical system generated by $T_0$ and $T_1$. This means that
$T_1$ is chosen with a fixed probability $p\in(0,1)$ at each step.
Let $m_p$ denote the Bernoulli measure on the sample space $\Omega$
associated with the probability vector $(1-p,p)$. 
  By \cite[Theorem~5.2]{Ino12}, 
there exists a unique 
Borel probability measure $\lambda_p$ on $[0,1]$ that is absolutely continuous with respect to the Lebesgue measure on $[0,1]$ and satisfies $\mu=(1-p)\cdot\mu\circ T_0^{-1}+p\cdot\mu \circ T_1^{-1}$.  
The measure $\lambda_p$, called the {\it Gauss-R\'enyi measure}, is significant since 
for $m_p$-almost every $\omega\in\Omega$ and Lebesgue almost every $x\in\Lambda_\omega$, we have
  \[\lim_{n\to\infty}\frac{1}{n}\sum_{i=0}^{n-1}f(T_\omega^ix)=\int fd\lambda_p\ \text{
 for any continuous $f\colon[0,1]\to\mathbb R$.}\]

 For $p\in[0,1)$, 
 let $h_p\colon[0,1]\to[0,\infty)$ denote the Radon-Nikod\'ym derivative 
 of $\lambda_p$  with respect to the Lebesgue measure on $[0,1]$. We know that $h_0(x)=\frac{1}{\log2}\frac{1}{x+1}$.
 For any $p\in(0,1)$, $h_p$
  is bounded from above and away from $0$ \cite[Proposition~3.4]{KKV17}.
 An explicit formula for $h_p$ is desired, since it is related to the frequency of digits in the random continued fraction expansion \eqref{random-s}. 
 Up to present, 
   no algebraic formula for $h_p$ is known except for the case $p=0$. Kalle et al. proved that $h_p$ is $C^\infty$ for any $p\in(0,1)$
 \cite{KMTV22}.
 Bahsoun et al. \cite{BRS20} obtained a functional-analytic formula for $h_p$ for $p\in(0,1)$ sufficiently near $0$. 

 Our aim here is to represent $\lambda_p$ and $h_p$ 
 for {\it any} $p\in(0,1)$, 
 using the collection of ``periodic points'' 
\[\bigcup_{\omega\in\Omega}\bigcup_{n=1}^\infty{\rm Fix}(T_\omega^n),\ \ {\rm Fix}(T_\omega^n)=\{x\in \Lambda_\omega\colon T_\omega^nx=x\}.\] 
 Elements of this set are called {\it random cycles} \cite{SuzTak21}. 
We first present a {\it  quenched} (samplewise) representation, and then an {\it annealed} (sample-averaged) one. 
 For $\omega\in\Omega$ and $n\in\mathbb N$ define
\begin{equation}\label{zomegan}Z_{\omega,n}=\sum_{x\in{\rm Fix}(T_{\omega }^n)}\!\!\!\!\!\!|(T_{\omega}^n)'x|^{-1},\end{equation}
which plays the role of a normalizing constant. 
 The derivatives of $T_0$ and $T_1$ at  their discontinuities are the one-sided derivatives.
For a topological space $X$, 
let $\mathcal M(X)$
denote the space of Borel probability measures on $X$ endowed with the weak* topology.
 For $\omega\in\Omega$, $x\in\Lambda_\omega$ and $n\in\mathbb N$, let $V_n^\omega(x)\in\mathcal M([0,1])$ denote the uniform probability distribution on the random orbit 
 $(T_\omega^ix)_{i=0}^{n-1}$.
 For $p\in\{0,1\}$,
let $m_p$ denote the Borel probability measure on $\Omega$ that is the unit point mass at the point
  $p^\infty=ppp\cdots$ in $\Omega$.
Let $\lambda_1\in\mathcal M([0,1])$ denote the unit point mass at $0$.

\begin{thm}[quenched representation of the Gauss-R\'enyi measure]\label{thm-a}
Let $p\in(0,1)$. The following statements hold:
 
 \begin{itemize}
\item[(a)] for $m_p$-almost every $\omega\in\Omega$ and any continuous function $F\colon \mathcal M([0,1])\to\mathbb R$, 
\[\lim_{n\to\infty}\frac{1}{Z_{\omega,n}}\sum_{x\in{\rm Fix}(T_\omega^n)}|(T^n_\omega)'x|^{-1}
F(V_n^{\omega}(x))=F(\lambda_p);\]

\item[(b)]  for $m_p$-almost every $\omega\in\Omega$ and any continuous function $f\colon [0,1]\to\mathbb R$, 
\[\lim_{n\to\infty} \frac{1}{Z_{\omega,n} }\sum_{x\in{\rm Fix}(T^n_\omega)}|(T_{\omega}^n)'x|^{-1}\int fdV_n^\omega(x)=\int f d\lambda_p.\]\end{itemize}
\end{thm}

As already noted,
the cases $p=0$ and $p=1$ correspond to the iteration of $T_0$ and that of $T_1$ respectively. The convergences in Theorem~\ref{thm-a} in these two cases
were established in \cite{Tak20} (see  \cite{FieFieYur02} for a closely related result) and
 \cite{Tak22} respectively. The main concern of this paper is the case $p\in(0,1)$.

 Theorem~\ref{thm-a}(a) implies
 Theorem~\ref{thm-a}(b) (see \S\ref{pf-thma}).
 The latter deserves to be called a quenched representation of $\lambda_p$ in terms of random cycles. 
 For $\omega\in\Omega$, $x\in\Lambda_\omega$, a subset $A$ of $[0,1]$ and $n\in\mathbb N$, let
\[e_n(\omega,x,A)=\frac{\#\{0\leq i\leq n-1\colon T_\omega^ix\in A\}}{n}.\]
By the portmanteau theorem, 
Theorem~\ref{thm-a}(b)
 is equivalent to the following:  for $m_p$-almost every $\omega\in\Omega$ and any Borel subset $A$ of $[0,1]$ with $\lambda_p(\partial A)=0$, 
 \begin{equation}\label{quench}\lim_{n\to\infty} \frac{1}{Z_{\omega,n} }\sum_{x\in{\rm Fix}(T^n_\omega)  }|(T_{\omega}^n)'x|^{-1}e_n(\omega,x,A)=\lambda_p(A).\end{equation}

The meaning of Theorem~\ref{thm-a}(a) may be a little less intuitive Theorem~\ref{thm-a}(b). By the portmanteau theorem it
 is equivalent to the following:  for for $m_p$-almost every $\omega\in\Omega$ and any Borel subset $\mathcal A$ of $\mathcal M(\Lambda)$ with $\lambda_p\notin\partial\mathcal A$, 
\[\lim_{n\to\infty}\frac{1}{Z_{\omega,n}}\sum_{\substack{x\in{\rm Fix}(T_\omega^n)\\ V_n^\omega(x)\in\mathcal A }}|(T^n_\omega)'x|^{-1}
=\1_{\mathcal A}(\lambda_p),\]
where $\1_{\mathcal A}$ denotes the indicator function of $\mathcal A$. In particular, if $\lambda_p\in\mathcal A$ then $V_n^\omega(x)\in\mathcal A$ holds for almost every $x\in{\rm Fix}(T_\omega^n)$
as $n\to\infty$.

To move on to an annealed counterpart, 
for $p\in[0,1]$, $n\in\mathbb N$ and 
$\omega\in\Omega$ we set
\[Z_{p,n}=\int Z_{\omega,n}dm_p(\omega),\]
which plays the role of a normalizing constant.
\begin{thm}[annealed representation of the Gauss-R\'enyi measure]
\label{thm-b}
Let $p\in(0,1)$. The following statements hold: 
\begin{itemize}
\item[(a)] for any continuous function $F\colon \mathcal M([0,1])\to\mathbb R$, 
\[\lim_{n\to\infty}\frac{1}{Z_{p,n}}\int dm_p(\omega)\sum_{x\in{\rm Fix}(T_\omega^n)}|(T_{\omega}^n)'x|^{-1}F(V_n^\omega(x))=F(\lambda_p);\]
\item[(b)] for any continuous function $f\colon [0,1]\to\mathbb R$,
\[\lim_{n\to\infty}\frac{1}{Z_{p,n}}\int dm_p(\omega)\sum_{x\in{\rm Fix}(T_{\omega }^n)}|(T_{\omega}^n)'x|^{-1}
\int fdV^{\omega}_n(x)= \int fd\lambda_p.\]
\end{itemize}
\end{thm}

Theorem~\ref{thm-b}(a) implies
Theorem~\ref{thm-b}(b) (see \S\ref{pf-thmb}).
 The latter deserves to be called an annealed representation of $\lambda_p$ in terms of random cycles since
 it is equivalent to the following:  for any Borel subset $A$ of $[0,1]$ with $\lambda_p(\partial A)=0$, 
 \begin{equation}\label{anneal}\lim_{n\to\infty} \frac{1}{Z_{p,n} }\int dm_p(\omega)\sum_{x\in{\rm Fix}(T^n_\omega)  }|(T_{\omega}^n)'x|^{-1}e_n(\omega,x,A)=\lambda_p(A).\end{equation}
 Theorem~\ref{thm-b}(a)
 is equivalent to the following:  for any Borel subset $\mathcal A$ of $\mathcal M(\Lambda)$ with $\lambda_p\notin\partial\mathcal A$, 
\[\lim_{n\to\infty}\frac{1}{Z_{p,n}}\int dm_p(\omega)\sum_{\substack{x\in{\rm Fix}(T_\omega^n)\\ V_n^\omega(x)\in\mathcal A }}|(T^n_\omega)'x|^{-1}
=\1_{\mathcal A}(\lambda_p).\]

Since the Radon-Nikod\'ym derivative $h_p$ of the Gauss-R\'enyi measure $\lambda_p$ is continuous,  
from \eqref{quench} and \eqref{anneal} 
we obtain its quenched and annealed representations in terms of random cycles.
 
\begin{cor}[quenched and annealed representations of the Radon-Nikod\'ym derivative]\label{cor1}
Let $p\in(0,1)$. The following statements hold:
\begin{itemize}
\item[(a)] for $m_p$-almost every $\omega\in\Omega$ and any $y\in(0,1)$,
 \[\label{recover}h_p(y)=\lim_{\varepsilon\to+0}\frac{1}{2\varepsilon}\lim_{n\to\infty}\frac{1}{Z_{\omega,n}}\sum_{x\in{\rm Fix}(T^n_\omega)}|(T_{\omega}^n)'x|^{-1}e_n(\omega,x,[y-\varepsilon,y+\varepsilon] );\]
\item[(b)] for any $y\in(0,1)$, 
\[h_p(y)=\lim_{\varepsilon\to+0}\frac{1}{2\varepsilon}\lim_{n\to\infty}\frac{1}{Z_{p,n}}\int dm_p(\omega)
\sum_{x\in{\rm Fix}(T_{\omega}^n)}|(T_\omega^n)'x|^{-1}e_n(\omega,x,[y-\varepsilon,y+\varepsilon]).\]
\end{itemize}
\end{cor}


 Our main results altogether assert that the collection of random cycles capture relevant information of the Gauss-R\'enyi random dynamics.
Since random cycles can be defined for general random dynamical systems, their
relevance in descriptions of random dynamical properties should be investigated in a much more broader context. Our main results support the relevance, while 
Buzzi \cite{Buz02} earlier proved that a dynamical zeta function defined with random cycles of certain random matrices cannot be extended beyond its disk of holomorphy, almost surely.
Under suitable assumptions, dynamical zeta functions of deterministic dynamical systems can be extended to meromorphic functions, and their zeros/poles are related to statistical properties of the underlying dynamics.
With our results including \cite{SuzTak21} and Buzzi's one \cite{Buz02} in mind, 
which information is captured by random cycles and which is not should be closely examined in the future.

\subsection{Method of proofs of the main results}
A basic strategy for proofs of our main results is to represent the i.i.d. random dynamical system generated by $T_0$ and $T_1$ as a skew product, and analyze the corresponding deterministic dynamical system.
Let
$\theta\colon\Omega\to\Omega$ denote the left shift:
$(\theta \omega)_n=\omega_{n+1}$ for $n\in\mathbb N$. 
Let
\[E=\{(\omega,x)\in\Omega\times[0,1]\colon(\omega_1,x)\in \{(0,0),(1,1)\}\},\]
and define $R\colon(\Omega\times[0,1])\setminus E\to \Omega\times[0,1]$ by
\[R(\omega,x)=(\theta\omega,T_{\omega_1}x).\]
Let
\[\Lambda=\bigcap_{n=0}^\infty R^{-n}\left((\Omega\times[0,1])\setminus E\right),\]
which is a non-compact set.
We still denote   $R|_{\Lambda}$ by $R$ and call it the {\it Gauss-R\'enyi map}. 
 We have
$R^n(\omega,x)=(\theta^n\omega,T_{\omega}^nx)$ for $(\omega,x)\in \Lambda$ and
 $n\in\mathbb N$, and so
\[\Lambda_\omega=\{x\in[0,1]\colon(\omega,x)\in\Lambda\}\] for every $\omega\in\Omega$.
For any $p\in[0,1]$, the map $R$ leaves invariant the Borel probability measure
$m_p\otimes\lambda_p$, the restriction of the product measure of $m_p$ and $\lambda_p$ to $\Lambda$.

For each $n\in\mathbb N$, let ${\rm Fix}(R^n)$ denote the set of periodic points of $R$ of period $n$.
A key observation is that $x\in{\rm Fix}(T_\omega^n)$ implies $(\omega',x)\in {\rm Fix}(R^n)$ where $\omega'\in\Omega$ is the repetition of the word $\omega_1\cdots\omega_n$ in $\omega$. For this reason,
 properties of random cycles may be analyzed through the analysis of periodic points of $R$.
Much of our effort is devoted to establishing annealed and quenched level-2 large deviations upper bounds for periodic points of $R$, and derive the desired convergences from the large deviations upper bounds.
 For $p\in[0,1]$, $n\in\mathbb N$ and $\omega\in\Omega$, define
  \[Q_{p}^n(\omega)=(1-p)^{\#\{1\leq i\leq n\colon \omega_i=0\}}p^{\#\{1\leq i\leq n\colon \omega_i=1\}},\]
  where we put $0^0=1$ for convenience.
  Notice that
\begin{equation}\label{zpn}Z_{p,n}= \sum_{(\omega,x)\in {\rm Fix}(R^n) }Q_p^n(\omega)|(T^{n}_{\omega})'x|^{-1}.\end{equation}
For $(\omega,x)\in\Lambda$
 and $n\in\mathbb N$, let
 $V_n^R(\omega,x)\in\mathcal M(\Lambda)$ denote
 the uniform probability distribution on the orbit $(R^i(\omega,x))_{i=0}^{n-1}$.
 Let  $\delta_{V_n^R(\omega,x)}$
  denote the Borel probability measure on $\mathcal M(\Lambda)$ that is the unit point mass at $V_n^R(\omega,x)$.
 Define a sequence $(\tilde\mu_n)_{n=1}^\infty$ of Borel probability measures on $\mathcal M(\Lambda)$ by
 \[\tilde\mu_n=\frac{1}{Z_{p,n}}
  \sum_{(\omega,x)\in {\rm Fix}(R^n) }Q_p^n(\omega)|(T^{n}_{\omega})'x|^{-1}\delta_{V_n^R(\omega,x)}.\]

\begin{thm}[annealed level-2 Large Deviation Principle]
\label{level-2-thm}
Let $p\in(0,1)$. The following statements hold:
\begin{itemize}
\item[(a)] $(\tilde\mu_n)_{n=1}^\infty$ is exponentially tight, and satisfies the LDP with the convex good rate function $I_p\colon\mathcal M(\Lambda)\to[0,\infty]:$
for any open subset $\mathcal G$ of $\mathcal M(\Lambda)$,
 \[\liminf_{n\to\infty}\frac{1}{n}\log \tilde\mu_{n}(\mathcal G)\geq -\inf_{\mathcal G} I_p,\] and
 for any closed subset $\mathcal C$ of $\mathcal M(\Lambda)$,
 \[\limsup_{n\to\infty}\frac{1}{n}\log \tilde\mu_{n}(\mathcal C)\leq -\inf_{\mathcal C} I_p.\]
The minimizer of $I_p$ is unique and it is $m_p\otimes\lambda_p$;
\item[(b)] 
for any bounded continuous function $F\colon \mathcal M(\Lambda)\to\mathbb R$, 
\[\lim_{n\to\infty}\frac{1}{Z_{p,n}}
  \sum_{(\omega,x)\in {\rm Fix}(R^n) }Q_p^n(\omega)|(T^{n}_{\omega})'x|^{-1}F(V_n^R(\omega,x))=F(m_p\otimes\lambda_p).\]
\end{itemize}
\end{thm}

 See
$\S$\ref{LDP-s} for the definition of the Large Deviation Principle and that of related terms in the statements of Theorem~\ref{level-2-thm}, including the meaning of level-2.
The statements in the cases $p=0$ and $p=1$
were established in \cite{Tak20} and
 \cite{Tak22} respectively. The main concern of this paper is the case $p\in(0,1)$.

Moving on to a quenched counterpart, for each $\omega\in\Omega$ we define
 a sequence $(\tilde\mu_n^\omega)_{n=1}^\infty$ of Borel probability measures on  $\mathcal M(\Lambda)$ by
  \[\tilde\mu_n^\omega=\frac{1}{Z_{\omega,n}}\sum_{ x\in{\rm Fix}(T^n_\omega) }|(T_{\omega}^n)'x|^{-1}\delta_{V_n^R(\omega,x)}.\]
The measure 
$\int_\Omega\tilde\mu^\omega_n(\cdot)dm_p(\omega)$ on $\mathcal M(\Lambda)$ equals $\tilde\mu_n(\cdot)$ up to subexponential factors (see Lemma~\ref{lem-cal1}). 

\begin{thm}[quenched level-2 large deviations]\label{ldpup-q}
Let $p\in(0,1)$. The following statements hold:
\begin{itemize}
\item[(a)] for $m_p$-almost every $\omega\in\Omega$,
$(\tilde\mu_n^\omega)_{n=1}^\infty$ is exponentially tight, and
 for any closed subset $\mathcal C$ of $\mathcal M(\Lambda)$,
 \[\limsup_{n\to\infty}\frac{1}{n}\log \tilde\mu_{n}^\omega(\mathcal C)\leq -\inf_{\mathcal C} I_p;\]
\item[(b)] for $m_p$-almost every $\omega\in\Omega$ and any bounded continuous function $F\colon \mathcal M(\Lambda)\to\mathbb R$, 
\[\lim_{n\to\infty}\frac{1}{Z_{\omega,n}}\sum_{ x\in{\rm Fix}(T^n_\omega) }|(T_{\omega}^n)'x|^{-1}F(V_n^R(\omega,x))=F(m_p\otimes\lambda_p).\]
\end{itemize}
\end{thm}


The rest of this paper consists of three sections.
In \S2 we prove Theorem~\ref{thm-a} and Theorem~\ref{thm-b} subject to Theorem~\ref{level-2-thm} and Theorem~\ref{ldpup-q}. These deductions are rather straightforward. In \S3
we start an analysis of the Gauss-R\'enyi map $R$,
and prove Theorem~\ref{ldpup-q} subject to 
Theorem~\ref{level-2-thm}. In \S4 we prove Theorem~\ref{level-2-thm}.

A more precise logical structure is indicated in the diagram below.
In \S\ref{pf-thmb} we show 
Theorem~\ref{level-2-thm}(b) $\Longrightarrow$ Theorem~\ref{thm-b}.
In \S\ref{pf-thma} we show 
Theorem~\ref{ldpup-q}(b) $\Longrightarrow$ Theorem~\ref{thm-a}.
 In \S\ref{pf-sample} we show
Theorem~\ref{level-2-thm}(a) $\Longrightarrow$ Theorem~\ref{ldpup-q}(a) $\Longrightarrow$ Theorem~\ref{ldpup-q}(b).



\[  { \begin{CD}
     {\bf Theorem~\ref{level-2-thm}(a)} @>{\S\ref{pf-sample}}>> {\bf Theorem~\ref{ldpup-q}(a)} \\
  @V{\S\ref{section4-1}}VV    @VV{\S\ref{pf-sample}}V \\
     {\bf Theorem~\ref{level-2-thm}(b) }  @.  
     {\bf Theorem~\ref{ldpup-q}(b)} \\
     @V{\S\ref{pf-thmb}}VV    @VV{\S\ref{pf-thma}}V \\
     {\bf Theorem~\ref{thm-b}}  @.  {\bf Theorem~\ref{thm-a} }\\
  \end{CD}}\]
  \\





        Most of our effort is dedicated to the proof of Theorem~\ref{level-2-thm}(a).
          The random dynamical system we consider falls into the class of {\it mean expanding systems} that are comprehensively investigated in \cite{ANV15}. Moreover,  
           the restriction of the Perron-Frobenius operator associated with the Gauss-R\'enyi map $R$ to an  appropriate function space has a spectral gap \cite{KKV17, KMTV22}. 
           This property can be used to apply the general results in \cite{ANV15} to deduce nice statistical properties of the dynamical system $(\Lambda,R,m_p\otimes\lambda_p)$, see \cite{KKV17} for details.
Meanwhile, 
it is not known whether the existence of spectral gap implies the LDP. 
To prove Theorem~\ref{level-2-thm}(a),
our strategy
is to code the Gauss-R\'enyi map 
into 
the countable full shift, 
 establish the LDP there, 
and then transfer this LDP back to the original system. 

Owing to the existence of the neutral fixed point of the R\'enyi map $T_1$,
for the potential function associated with this countable full shift there exists no Gibbs state. 
To resolve this difficulty, we construct an appropriate induced system that is topologically conjugate to another countable full shift, 
and then apply the result of the second-named author in \cite{Tak22}. This requires verifying the regularity of the associated induced potential.
The use of induced systems for an analysis of random dynamical systems with infinitely many branches can be found, for example, in \cite{DaOo17}.

The uniqueness of minimizer in Theorem~\ref{level-2-thm}(a) is important to ensure the convergence in Theorem~\ref{level-2-thm}(b). To establish this uniqueness, we first show the uniqueness of equilibrium state (see Proposition~\ref{unique-equi}), and then show that any minimizer is an equilibrium state. The first step relies on implementing the thermodynamic formalism for countable Markov shifts (see e.g., \cite{MauUrb03,Sar99}) with the induced system. Except for the construction of induced system and the verification of regularity of induced potential, the argument follows well-known lines (see e.g., \cite{MauUrb03,PesSen08}). In the second step we appeal to the result of the second named author \cite{Tak20}. 


\section{Deduction of 
convergences on random cycles}
As a warm up, 
in $\S$\ref{random-s} we begin by
describing an induction algorithm that generates random continued fractions. 
In $\S$\ref{LDP-s} we summarize basic facts on large deviations.
We show Theorem~\ref{level-2-thm}(b) $\Longrightarrow$ Theorem~\ref{thm-b} and 
Theorem~\ref{ldpup-q}(b) $\Longrightarrow$ Theorem~\ref{thm-a}, respectively in \S\ref{pf-thmb} and \S\ref{pf-thma}.
Those readers who would like to immediately access the proofs of Theorems~\ref{thm-a} and \ref{thm-b}  
can pass \S\ref{random-s}, \S\ref{LDP-s} and directly go to \S\ref{pf-thmb} and \S\ref{pf-thma}.
\medskip

\noindent{\it Notation.} For a bounded interval $J$, let $|J|$ denote its Euclidean length. 

\subsection{A continued fraction algorithm by the Gauss-R\'enyi map}\label{random-s}

Using the Gauss-R\'enyi map, we describe an induction algorithm generating random continued fractions.
 Define a function $C\colon (\Omega\times[0,1])\setminus E\to\mathbb N$ by 
\[C(\omega,x)=\left\lfloor \frac{1}{(-1)^{\omega_1}x+\omega_1}\right\rfloor.\]
For $(\omega,x)\in(\Omega\times [0,1])\setminus E$ and $n\in\mathbb N$, let
  \[C_n(\omega,x)=C(R^{n-1}(\omega,x))+\omega_{n+1},\]
  when $R^{n-1}(\omega,x)$ is defined.

 For any $(\omega,x)\in(\Omega\times[0,1])\setminus E$ we have
\[
x=\omega_1+\frac{(-1)^{\omega_1}}{C(\omega,x)+
T_{\omega_1}x}.\]
If  $R(\omega,x)\notin E$, then
replacing $(\omega,x)$ in \eqref{first} by $R(\omega,x)$ we have
\[
T_{\omega_1}x=\omega_2+\frac{(-1)^{\omega_2}}{C(R(\omega,x))+
T_{\omega}^2x}.\]
Substituting this into the right-hand side of the previous equality yields 
\[
x = \omega_1+\confrac{(-1)^{\omega_1}}{C(\omega,x)+\omega_2}  + \confrac{(-1)^{\omega_{2}}}
{C(R(\omega,x))+ T_\omega^{2}x}.
\]
If $n\geq2$ and
$R^i(\omega,x)\notin E$ for $i=0,\ldots,n-1$,
then repeating the above process yields 
\[
x = \omega_1+\confrac{(-1)^{\omega_1}}{C_1(\omega,x)}  + \cdots +\confrac{(-1)^{\omega_{n-1}}}{C_{n-1}(\omega,x)}+ 
\confrac{(-1)^{\omega_{n}}}
{C_{n}(\omega,x)-\omega_{n+1}+ T_\omega^{n}x},
\]
where $(-1)^{\omega_{i+1}}+C_{i}(\omega,x)\geq1$
  for $i=1,\ldots,n$. 
  
  For many $(\omega,x)$,
  this algorithm produces a continued fraction expansion of $x$ summarized as follows.

 \begin{prop}\label{random-lem}Let $(\omega,x)\in(\Omega\times[0,1])\setminus E$.
  \begin{itemize}
\item[(a)] If $x\in\Lambda_\omega$, then $(-1)^{\omega_{n+1}}+C_{n}(\omega,x)\geq1$ for every $n\in\mathbb N$, and the continued fraction
  \[ \omega_1+\confrac{(-1)^{\omega_1}}{C_1(\omega,x)}  +\confrac{(-1)^{\omega_{2}}}{C_{2}(\omega,x)}+\confrac{(-1)^{\omega_3}}{C_3(\omega,x)}+\cdots
\] converges to $x$. 
\item[(b)] If $x\in\Lambda_\omega$, then $x\notin\mathbb Q$ 
if and only if $(-1)^{\omega_{n+1}}+C_{n}(\omega,x)\geq2$ for infinitely many $n\in\mathbb N$.
\item[(c)] If $x\notin\Lambda_\omega$ then $x\in\mathbb Q$.
\end{itemize}
 \end{prop}

To prove (a) and (b) we use the next lemma. For related results, see \cite{Kra91,Per50,Tie11}.

\begin{lemma}[{
\cite[Lemma~2.1(a)]{NT24}}]\label{IFS-new} 
 Let $\omega\in\Omega$ and
 $(C_n)_{n\in\mathbb N}\in \mathbb N^{\mathbb N}$ satisfy $(-1)^{\omega_{n+1}}+C_{n}\geq1$ for every $n\in\mathbb N$.
Then the continued fraction
\[\omega_1+
\confrac{(-1)^{\omega_1} }{C_{1}} + \confrac{(-1)^{\omega_2} }{C_{2}}  +\confrac{(-1)^{\omega_3} }{C_{3}}+\cdots
\]
converges to a number in $[0,1]$. This number is irrational 
if and only if $(-1)^{\omega_{n+1}}+C_{n}\geq2$ for infinitely many $n\in\mathbb N$.
\end{lemma}
 
 \begin{proof}[Proof of Proposition~\ref{random-lem}]
 Let $x\in\Lambda_\omega$.  
Applying the algorithm to $(\omega,x)$ we get \begin{equation}\label{first}
x=\omega_1+\frac{(-1)^{\omega_1}}{C(\omega,x)+
T_{\omega_1}x},\end{equation}
and for every $n\geq2$,
\begin{equation}\label{CF-n}
x = \omega_1+\confrac{(-1)^{\omega_1}}{C_1(\omega,x)}  + \cdots +\confrac{(-1)^{\omega_{n-1}}}{C_{n-1}(\omega,x)}+ 
\confrac{(-1)^{\omega_{n}}}
{C_{n}(\omega,x)-\omega_{n+1}+ T_\omega^{n}x},
\end{equation}
where $(-1)^{\omega_{i+1}}+C_{i}(\omega,x)\geq1$
  for $i=1,\ldots,n$.
By Lemma~\ref{IFS-new}, 
the continued fraction \[\omega_1+\confrac{(-1)^{\omega_1}}{C_1(\omega,x)}  +\confrac{(-1)^{\omega_{2}}}{C_{2}(\omega,x)}+\confrac{(-1)^{\omega_{3}}}{C_{3}(\omega,x)}+\cdots\] converges to a number $y\in[0,1]$. Moreover, $y\notin\mathbb Q$
if and only if $(-1)^{\omega_{n+1}}+C_{n}(\omega,x)\geq2$ for infinitely many $n\in\mathbb N$. Hence,
for (a) and (b) it suffices to show $x=y$.

For each $n\in\mathbb N$, let $J_n(\omega,x)$ denote the maximal subinterval of $[0,1]$ containing $x$ on which $T_\omega^n$ is monotone. From \eqref{CF-n} we have $y\in J_n(\omega,x)$ for every $n\in\mathbb N$.
Since
$(-1)^{\omega_{n+1}}+C_{n}(\omega,x)\geq1$,
there are four cases: 

\begin{itemize}\item[(i)] $\omega_n=\omega_{n+1}=0$; \item[(ii)] $\omega_n=1$ and $\omega_{n+1}=0$; \item[(iii)] $\omega_n=0$, $C(R^{n-1}(\omega,x))\geq2$ and $\omega_{n+1}=1$; \item[(iv)] $\omega_n=\omega_{n+1}=1$.
\end{itemize}
We estimate the derivatives of the composition
using the definitions of $T_0$ and $T_1$, $\inf_{(0,1]} |T'_0|\geq1$ and
$\inf_{[0,1)} |T'_1|\geq1$,
the monotonicity of $|T_0|$ on $(0,1]$ and that of   $|T_1'|$ on $[0,1)$. In case (i), for all $y\in T_\omega^{n-1}J_n(\omega,x)$ we have 
\[|(T_{\omega_{n+1}}\circ T_{\omega_n})'y|\geq 
\left|T_0'\left(\frac{2}{3}\right)\right|=\frac{9}{4}.\]
In case (ii),  for all $y\in T_\omega^{n-1}J_n(\omega,x)$  we have 
\[|(T_{\omega_{n+1}}\circ T_{\omega_n})'y|\geq 
\left|T_1'\left(\frac{1}{3}\right)\right|=\frac{9}{4}.\]
In case (iii), for all $y\in T_\omega^{n-1}J_n(\omega,x)$ we have 
\[|(T_{\omega_{n+1}}\circ T_{\omega_n})'y|\geq 
\left|T_0'\left(\frac{1}{2}\right)\right|>\frac{9}{4}.\]
Hence, if one of (i) (ii) (iii) occurs infinitely many times then 
$\inf_{J_n(\omega,x)}|(T_\omega^n)'|\to\infty$ as $n\to\infty$.
By the mean value theorem, for every $n\in\mathbb N$
there exists $\xi_n\in J_n(\omega,x)$ such that
\[|x-y|=\frac{|T_\omega^nx-T_\omega^ny|}{|(T_\omega^n)'\xi_n|}\leq\frac{1}{|(T_\omega^n)'\xi_n|}.\]
Letting $n\to\infty$ we obtain $x=y$.

If all (i) (ii) (iii) occur only finitely many times, then there is $k\in\mathbb N$ such that $\omega_n=1$ for every $n> k$.
Suppose $T_\omega^kx\notin\mathbb Q$. Then $T_1^n(T_\omega^kx)\neq0$ holds for every $n\in\mathbb N$.
Then the formula for $T_1$ implies
$\inf_{J_{n-k}(1^\infty,T_\omega^kx)}|(T_1^{n-k})'|\to\infty$ as $n\to\infty$.
For every $n\in\mathbb N$
there exists $\zeta_n\in J_{n-k}(1^\infty,T_\omega^kx)$ such that
\[|T_\omega^kx-T_\omega^ky|=\frac{|T_\omega^nx-T_\omega^ny|}{|(T_1^{n-k})'\zeta_n|}\leq\frac{1}{|(T_1^{n-k})'\zeta_n|}.\]
Letting $n\to\infty$ we obtain
$T_\omega^kx=T_\omega^ky$. Since the restriction of $T_\omega^k$ to $J_k(\omega,x)$ is injective, we obtain $x=y$.
Suppose $T_\omega^kx\in\mathbb Q$.
Since $T_1$ maps all rational points to $0$,
 there exists $n\in\mathbb N$ such that $T_1^n(T_\omega^kx)=0$. Since the neutral fixed point $0$ of $T_1$ is topologically repelling, it follows that $T_1^n(T_\omega^ky)=0$. The restriction of $T_\omega^{k+n}$ to $J_{k+n}(\omega,x)$ is injective, 
 and hence $x=y$.
We have verified (a) and (b).

If $x\in(0,1)\setminus\Lambda_\omega$ then there exists $n\in\mathbb N$ such that $T^{n}_\omega x$ is defined and $T^{n+1}_\omega x$ is not defined. 
Then $T^{n}_\omega x\in \{0,1\}$ holds and \eqref{first}, \eqref{CF-n} together imply $x\in\mathbb Q$, verifying (c). The proof of Proposition~\ref{random-lem} is complete. \end{proof}

\subsection{Large Deviation Principle}\label{LDP-s}
Our main reference on large deviations is \cite{DemZei98}. Let $\mathcal X$ be a topological space and let $(\mu_n)_{n=1}^\infty$ be a sequence of Borel probability measures on $\mathcal X$.
We say the {\it Large Deviation Principle} (LDP) holds for 
 $(\mu_n)_{n=1}^\infty$  
  if there exists a lower semicontinuous function $I\colon\mathcal X\to[0,\infty]$ such that:
  \begin{itemize}
\item[(a)] for any open subset 
$\mathcal G$ of $\mathcal X$,
 \[\liminf_{n\to\infty}\frac{1}{n}\log \mu_n(\mathcal G)\geq -\inf_{\mathcal G} I;\]
\item[(b)]  for any closed subset $\mathcal C$ of $\mathcal X$,
\[\limsup_{n\to\infty}\frac{1}{n}\log \mu_n(\mathcal C)\leq-\inf_{\mathcal C}I.\]
\end{itemize}
We say $x\in \mathcal X$ is a {\it minimizer} if $I(x)=0$ holds. 
 The LDP roughly means that in the limit $n\to\infty$ the measure $\mu_n$ assigns
      all but exponentially small mass 
      to the set $\{x\in\mathcal X\colon I(x)=0\}$
      of minimizers.
      The function $I$ is called a {\it rate function}.
If $\mathcal X$ is a metric space and $(\mu_n)_{n=1}^\infty$ satisfies the LDP, the rate function is unique. 
We say the rate function $I$ is {\it good} if the set $\{x\in\mathcal X\colon I(x)\leq c\}$ is compact for any $c>0$.

We say $(\mu_n)_{n=1}^\infty$ is {\it exponentially tight} if for any $L>0$ there exists a compact subset $\mathcal K$ of $\mathcal X$ such that
\[\limsup_{n\to\infty}\frac{1}{n}\log \mu_n(\mathcal X \setminus \mathcal K)\leq-L.\]
If $(\mu_n)_{n=1}^\infty$ is exponentially tight then it is tight, i.e., for any $\varepsilon>0$ there exists a compact subset $\mathcal K'$ of $\mathcal X$ such that $\mu_n(\mathcal K')>1-\varepsilon$ for all sufficiently large $n$.

 \begin{prop}\label{CP}Let $\mathcal X$, $\mathcal Y$ be Hausdorff spaces and let $(\mu_n)_{n=1}^\infty$ be a sequence of Borel probability measures on $\mathcal X$ for which the LDP holds with a good rate function $I$. Let $f\colon \mathcal X\to \mathcal Y$ be a continuous map. Then the LDP holds for $(\mu_n\circ f^{-1})_{n=1}^\infty$ with a good rate function $J\colon \mathcal Y\to [0,\infty]$ given by  \[J(y)=\inf\{I(x)\colon x\in \mathcal X,\ f(x)=y\}.\]
Moreover, 
if $y_0\in \mathcal Y$ is a mininizer of $J$, then there is a minimizer $x_0\in\mathcal X$ of $I$ such that $y_0=f(x_0)$. \end{prop}
The first assertion of Proposition~\ref{CP} is well-known as {\it the Contraction Principle}. 
Here we only include a proof of the second assertion.

\begin{proof}[Proof of the second assertion of Proposition~\ref{CP}]
Let $y_0\in\mathcal Y$ be a minimizer of $J$. By the definition of $J$,
there is a sequence $(x_n)_{n=1}^\infty$ in $\mathcal X$ such that $y_0=f(x_n)$
and $I(x_n)<1/n$ for every $n\geq1$. Since $I$ is a good rate function, $(x_n)_{n=1}^\infty$
has a limit point, say $x_0$. Since $I$ is lower semicontinuous, $x_0$ is a minimizer of $I$. 
 Since $f$ is continuous, we obtain $y_0=f(x_0)$.
\end{proof}

Let $X$ be a topological space and let $C(X)$ denote the Banach space of real-valued bounded continuous functions on $X$ endowed with the supremum norm. Recall that
the {\it weak* topology} on $\mathcal M(X)$ is the coarsest topology that makes the map $\mu\in\mathcal M(X)\mapsto\int fd\mu$ continuous for any  $f\in C(X)$. In this topology, a sequence $(\mu_n)_{n=1}^\infty$ of elements of $\mathcal M(X)$ converges to $\mu\in\mathcal M(X)$ if and only if $\lim_{n}\int fd\mu_n=\int fd\mu$ holds for any $f\in C(X)$. This condition is equivalent to $\lim_{n}\int fd\mu_n=\int fd\mu$ for any $f\in C(X)$ that is uniformly continuous  (see \cite[Chapter~9]{Str25}).

Donsker and Varadhan have identified three levels of the LDP, see e.g.,  \cite[Chapter~I]{Ell85}.
The LDP for a sequence of Borel probability measures on $\mathcal M(X)$ is referred to as {\it level-2}. The LDP for a sequence of Borel probability measures on $\mathbb R$ determined by a real-valued function on $X$ is referred to
as {\it level-1}. By the Contraction Principle, any level-2 LDP can be transferred to a level-1 LDP. 
\medskip

\noindent{\it Notation.}
For a topological space $X$, let $\mathcal M^2(X)$ denote the space of Borel probability measures on $\mathcal M(X)$ endowed with the weak* topology.
 For each $\mu\in\mathcal M(X)$, let $\delta_\mu\in\mathcal M^2(X)$ denote the unit point mass at $\mu$.

\subsection{Proof of Theorem~\ref{thm-b}}\label{pf-thmb}

 We define a sequence $(\tilde\xi_n)_{n=1}^\infty$ in $\mathcal M^2([0,1])$ by
\[\tilde\xi_n=\frac{1}{Z_{p,n}}\int dm_p(\omega)\sum_{x\in{\rm Fix}(T_\omega^n)}|(T_{\omega}^n)'x|^{-1}\delta_{V_n^\omega(x)}.\]
Also, we define a sequence $(\xi_n)_{n=1}^\infty$ in $\mathcal M([0,1])$ by
\[\xi_n=\frac{1}{Z_{p,n}}\int dm_p(\omega)\sum_{x\in{\rm Fix}(T_{\omega }^n)}|(T_{\omega}^n)'x|^{-1}
V^{\omega}_n(x).\]
The convergence in Theorem~\ref{thm-b}(a) is equivalent to the convergence of
$(\tilde{\xi}_n)_{n=1}^\infty$ to $\delta_{\lambda_p}$
in $\mathcal M^2(\Lambda)$.
The convergence in Theorem~\ref{thm-b}(b) is equivalent to the convergence of
$(\xi_n)_{n=1}^\infty$ to $\lambda_p$
in $\mathcal M([0,1])$.

Let $\Pi\colon\Omega\times[0,1]\to[0,1]$ be the projection to the second coordinate. 
The restriction of $\Pi$
to $\Lambda$ induces a continuous map $\Pi_*\colon\mu\in\mathcal M(\Lambda)\mapsto\mu\circ\Pi^{-1}\in\mathcal M([0,1])$, which induces a continuous map 
$\tilde\mu\in \mathcal M^2(\Lambda )\mapsto\tilde\mu\circ\Pi_*^{-1}\in\mathcal M^2([0,1])$.
Note that
$\Pi_*(\mu)=\nu$ implies
$\delta_\mu\circ\Pi_*^{-1}=\delta_\nu.$
In particular, $\delta_{m_p\otimes\lambda_p}\circ\Pi_*^{-1}=\delta_{\lambda_p}$ and
$\delta_{V_n^R((\omega,x)) }   \circ\Pi_*^{-1}=\delta_{V_n^\omega(x) }$ for 
$(\omega,x)\in{\rm Fix}(R^n)$, and the latter yields
 $\tilde\mu_n\circ\Pi_*^{-1}
=\tilde\xi_n.$
By Theorem~\ref{level-2-thm}(b),  
$(\tilde\mu_n)_{n=1}^\infty$ converges 
to $\delta_{m_p\otimes\lambda_p}$ in $\mathcal M^2(\Lambda)$, and hence 
$(\tilde\xi_n)_{n=1}^\infty$ converges to
$\delta_{\lambda_p}$ in $\mathcal M^2([0,1])$
as required in Theorem~\ref{thm-b}(a).

We define a continuous map
$\Xi\colon \mathcal M^2([0,1])\to \mathcal M([0,1])$ as follows. 
For each $\tilde\mu\in \mathcal M^2([0,1])$,
consider the positive normalized bounded linear functional on $C([0,1])$ given by
\[f\in C([0,1])\mapsto\int \left(\int f d\mu\right)d\tilde\mu(\mu).\]
Using Riesz's representation theorem, we define $\Xi(\tilde\mu)$ to be the unique element of $\mathcal M([0,1])$ such that
\[\int fd\Xi(\tilde\mu)=\int \left(\int 
fd\mu\right)d\tilde\mu(\mu)\ \text{ for all $f\in C([0,1])$.}\]
Clearly $\Xi$ is continuous, satisfies 
$\Xi(\tilde\xi_{n})=\xi_{n}$ for every $n\in\mathbb N$
and $\Xi(\delta_{\lambda_p})=\lambda_p$. Hence, Theorem~\ref{thm-b}(b) follows from Theorem~\ref{thm-b}(a). \qed


\subsection{Proof of Theorem~\ref{thm-a}}\label{pf-thma}
For each $\omega\in\Omega$, 
define a sequence $({\xi}^\omega_n)_{n=1}^\infty$ in $\mathcal M^2([0,1])$ by  
\[\tilde{\xi}^\omega_n=\frac{1}{Z_{\omega,n}}\sum_{ x\in{\rm Fix}(T^n_\omega)} |(T_{\omega}^n)'x|^{-1}\delta_{V_n^\omega(x)}.\]
Also, define a sequence $({\xi}^\omega_n)_{n=1}^\infty$ in $\mathcal M([0,1])$ by
\[{\xi}^\omega_n=\frac{1}{Z_{\omega,n}}\sum_{x\in{\rm Fix}(T^n_\omega)}|(T_{\omega}^n)'x|^{-1}V_{n}^{\omega}(x).\]
The convergence in Theorem~\ref{thm-a}(a) is equivalent to the convergence of
$(\tilde{\xi}^\omega_n)_{n=1}^\infty$ to $\delta_{\lambda_p}$
in $\mathcal M^2([0,1])$.
The convergence in Theorem~\ref{thm-a}(b) is equivalent to the convergence of
$(\xi^\omega_n)_{n=1}^\infty$ to $\lambda_p$ 
in $\mathcal M([0,1])$.

To finish,
we trace the proof of Theorem~\ref{thm-b}.
By Theorem~\ref{ldpup-q}(b), $(\tilde\mu_n^\omega)_{n=1}^\infty$ converges to $\delta_{m_p\otimes\lambda_p}$ in  $\mathcal M^2(\Lambda)$. Since
 $\tilde\mu_n^\omega\circ\Pi_*^{-1}
=\tilde\xi_n^\omega$, 
$(\tilde\xi_n^\omega)_{n=1}^\infty$ converges to
$\delta_{\lambda_p}$ in $\mathcal M^2([0,1])$
as required in Theorem~\ref{thm-a}(a).
Since
$\Xi(\tilde\xi_{n}^\omega)=\xi_{n}^\omega$ and   $\Xi(\delta_{\lambda_p})=\lambda_p$,
$(\xi_n^\omega)_{n=1}^\infty$ converges to
$\lambda_p$ in $\mathcal M([0,1])$
as required in Theorem~\ref{thm-a}(b).
  \qed

\section{Fundamental analysis of the Gauss-R\'enyi map}
In this section we start  
the analysis of the Gauss-R\'enyi map $R$.
In \S\ref{induce-sect} we introduce an inducing scheme and some related objects.
In \S\ref{expansion-sec} we introduce an induced map $\widehat R$ and investigate its expansion properties. In \S\ref{anneal-sec} we introduce an annealed geometric potential $\varphi$ and evaluate distortions of its Birkhoff averages.
In \S\ref{subexp} we prove several preliminary lemmas needed for the proof of Theorem~\ref{ldpup-q}. The proof of Theorem~\ref{ldpup-q} is given in \S\ref{pf-sample}.\medskip

 \noindent{\it Convention.} Since $p\in(0,1)$ is a fixed constant for the rest of the paper, it will be mostly omitted from each statement.

\subsection{Inducing scheme}\label{induce-sect}
An {\it inducing scheme} of 
a dynamical system $T\colon X\to X$
is a pair $(Y,t_Y)$, where $Y$ is a proper subset of $X$ 
and $t_Y\colon Y\to\mathbb N\cup\{\infty\}$ is a function given by 
\[t_Y(x)=\inf\{n\geq1\colon T^nx\in Y\}.\]
Given an inducing scheme $(Y,t_Y)$ of $T\colon X\to X$, 
for each $k\in\mathbb N$ we set \[\{t_Y=k\}=\{x\in Y\colon t_Y(x)=k\},\] and
 define an {\it induced map} 
\[\widehat T\colon \bigcup_{k=1}^\infty\{t_Y=k\}\mapsto {\widehat T}^{
t_Y(x)}x\in Y,\]
 and define an {\it inducing domain} 
\[ \widehat X=\bigcap_{n=0}^\infty {\widehat T}^{-n}\left(\bigcup_{k=1}^\infty\{t_Y=k\}\right).\]
In other words, $t_Y$ is the first return time to $Y$, $\widehat T$ is the first return map to $Y$ and $\widehat X$
 is the domain on which $\widehat T$ can be iterated infinitely many times.
We still denote by $\widehat T$ the restriction of $\widehat T$ to $\widehat X$.
We call 
 $\widehat T\colon\widehat X\to\widehat X$ an {\it induced system} associated with the inducing scheme $(Y,t_Y)$.

 We will consider an induced system of the Gauss-R\'enyi map $R\colon\Lambda\to\Lambda$ and its symbolic version.
 We will attach the symbol `` $\widehat\cdot$ '' to denote objects associated with inducing schemes.

\subsection{Building uniform expansion}\label{expansion-sec}
Let $\mathbb N_0$ and $\mathbb N_1$ denote the sets of even and odd positive integers respectively. A direct calculation shows that both $T_0$ and $T_1$ satisfy R\'enyi's condition, namely
 \begin{equation}\label{ren}
 \sup_{\left(\frac{2}{k+2},\frac{2}{k}\right]}\frac{|T_0''|}{|T_0'|^2}\leq2\quad\text{for all }k\in\mathbb N_0 \ \text{ and }\  \sup_{\left[\frac{k-1}{k+1},\frac{k+1}{k+3}\right) }\frac{|T_1''|}{|T_1'|^2}\leq2\quad\text{for all }k\in\mathbb N_1.
 \end{equation}

Define  $a_1\colon(\Omega\times[0,1])\setminus E\to\mathbb N$ by
\begin{equation}\label{a1-def}a_1(\omega,x)=\begin{cases}\vspace{1mm}k\in\mathbb N_0 &\text{ if }\omega_1=0\text{ and }x\in \displaystyle{\left(\frac{2}{k+2},\frac{2}{k}\right]},\\ k \in\mathbb N_1&\text{ if }\omega_1=1\text{ and }x\in \displaystyle{\left[\frac{k-1}
{k+1},\frac{k+1}{k+3}\right)}.\end{cases}\end{equation}
For each $(\omega,x)$ and $n\in\mathbb N$ such that $R^{n-1}(\omega,x)$ is defined, let
\[a_n(\omega,x)=a_1(R^{n-1}(\omega,x)).\]
 For $n\in\mathbb N$ and $a_1\cdots a_n\in \mathbb N^n$, define an {\it $n$-cylinder}
\[\varDelta(a_1\cdots a_n)=
\{(\omega,x)\in(\Omega\times[0,1])\setminus E\colon a_i(\omega,x)=a_i\text{ for }i=1,\ldots,n\}.\]
Let $\Pi\colon\Omega\times[0,1]\to[0,1]$ denote the projection to the second coordinate. 
We write $J(a_1\cdots a_n)$ for $\Pi(\varDelta(a_1\cdots a_n))$.
If $(\omega,x)\in \varDelta(a_1\cdots a_n)$ then $J(a_1\cdots a_n)$ is the maximal subinterval of $[0,1]$ containing $x$ on which $T_\omega^n$ is monotone.
The collection of $1$-cylinders defines a Markov partition for $R$: for every $k\in\mathbb N$, $R$ maps $\varDelta(k)$ bijectively onto its image and $R(\varDelta(k))$ contains $\Omega\times(0,1)$.


 \begin{figure}
\begin{center}
\includegraphics[height=4cm,width=7cm]{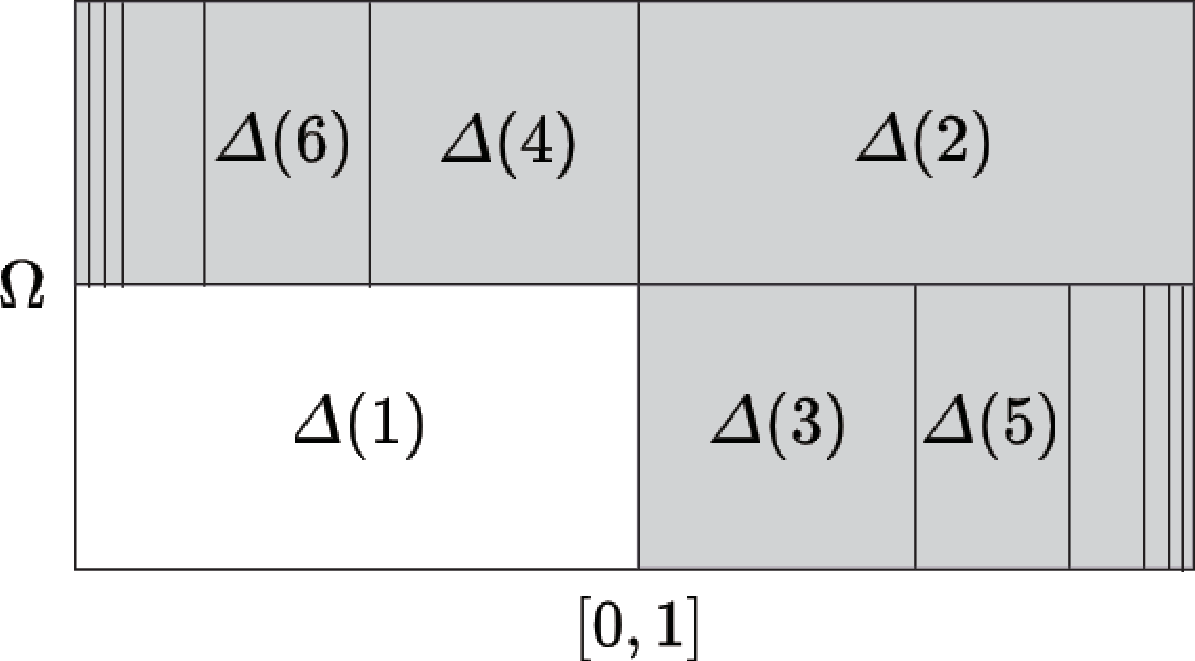}
\caption
{
The inducing domain $\widehat\Lambda$ associated with the inducing scheme $(\Lambda\setminus\varDelta(1),t_{\Lambda\setminus\varDelta(1)})$ is contained in $\bigcup_{k=2}^\infty\varDelta(k)$, the shaded area.}\label{induce-fig}
\end{center}
\end{figure}

Put 
\begin{equation}\label{omega0}\Omega_0=\{(\omega_n)_{n\in\mathbb N}\in\Omega\colon\omega_n=0\text{ for infinitely many $n$}\}.\end{equation}
Due to the presence of the neutral fixed point of the R\'enyi map $T_1$, the random composition of $T_0$ and $T_1$ is not uniformly expanding in that \[\inf_{\omega\in\Omega_0 }\inf_{\Lambda_\omega}\liminf_{n\to\infty}\frac{1}{n}\log|(T_\omega^n)'|=0.\]
To control the effect of the neutral fixed point, we consider the inducing scheme $(\Lambda\setminus\varDelta(1),t_{\Lambda\setminus\varDelta(1)})$ of
 $R\colon\Lambda\to\Lambda$
and the associated induced system $\widehat R\colon\widehat\Lambda\to\widehat\Lambda$, see \textsc{Figure}~\ref{induce-fig}.  
Let us abbreviate $t_{\Lambda\setminus\varDelta(1)}$ as $t$.
Note that $t(\omega,x)$ is finite if and only if $T_\omega x\neq0$.
The next lemma implies that the induced map $\widehat R$ is still not uniformly expanding. However, the lemma after the next one implies that $\widehat R^2$ is uniformly expanding.
\begin{lemma} 
Let $\omega\in\Omega$ satisfy $\omega_1=0$, $\omega_2=1$, $\omega_3=0$. Then we have \[\inf_{x\in\varDelta(2)}|(T_\omega^{t(\omega,x)})'x|=1.\]\end{lemma}
\begin{proof}
Since $\inf_{(0,1]}|T_0'|\geq1$ and $\inf_{[0,1)}|T_1'|\geq1$, we have 
$\inf_{x\in\varDelta(2)}|(T_\omega^{t(\omega,x)})'x|\geq1$.
By the hypothesis on $\omega$ and $T_01=0$, we have $\lim_{x\to1-0}t(\omega,x)=2$.
Using this and the monotonicity of $|T_0'|$ on $\varDelta(2)$ and that of $|T_1'|$ on $\varDelta(1)$, we obtain
 $\inf_{x\in\varDelta(2)}|(T_\omega^{t(\omega,x)})'x|\leq\lim_{x\to1-0}|(T_1\circ T_0)'x|=1$.
\end{proof}
\begin{lemma}\label{3-exp}
If $(\omega,x)\in\Lambda\setminus\varDelta(1)$, 
$t(\omega,x)$ and $t(\widehat R(\omega,x))$ are finite 
and $a_i(\omega,x)=a_i(\varrho,y)$ for $i=1,\ldots,t(\omega,x)+t(\widehat R(\omega,x))$, then
\[|(T_\omega^{t(\omega,x)+t(\widehat R(\omega,x)) })'y|\geq|(T_\omega^{t(\omega,x)+t(\widehat R(\omega,x))-1})'(T_\omega y)|\geq\frac{9}{4}.\] \end{lemma}
\begin{proof} 
From the definitions of $T_0$ and $T_1$, $\inf_{(0,1]} |T'_0|\geq1$,
$\inf_{[0,1)} |T'_1|\geq1$,
the monotonicity of $|T_0|$ on $(0,1]$ and that of   $|T_1'|$ on $[0,1)$,
if $(\omega,x)\notin\varDelta(2)$ then 
\[(T_\omega^{t(\omega,x)+t(\widehat R(\omega,x)) })'y|\geq |T_{\omega_1}'y| \geq\left|T_0'\left(\frac{1}{2}\right)\right|>\frac{9}{4}.\]
If $(\omega,x)\in\varDelta(2)$ and $T_\omega^{t(\omega,x)}x\in[1/2,1)$
 then 
\[(T_\omega^{t(\omega,x)+t(\widehat R(\omega,x)) })'y|\geq |T_{t(\omega,x)}'y| \geq\left|T_1'\left(\frac{1}{3}\right)\right|=\frac{9}{4}.\]
If $(\omega,x)\in\varDelta(2)$ and $T_\omega^{t(\omega,x)}x\in(0,1/2)$ 
 then 
\[(T_\omega^{t(\omega,x)+t(\widehat R(\omega,x)) })'y|\geq |T'(T_\omega^{t(\omega,x)}y)| \geq\left|T_0'\left(\frac{1}{2}\right)\right|>\frac{9}{4}.\]
Hence the desired inequality holds.
\end{proof}

\begin{lemma}[Uniform decay of cylinders]
\label{u-decay}
There exists $K\geq1$ such that
for every $n\in\mathbb N$ and every $a_1\cdots a_n\in \mathbb N^n$,
\[|J(a_1\cdots a_n)|\leq\frac{K}{\sqrt{n}}.\]
\end{lemma}
\begin{proof}
Take an integer $M\geq4$ such that for every $n\geq M$,
\begin{equation}\label{decay1}\left(\frac{9}{4}\right)^{-\sqrt{n}/2+1}\leq\frac{1}{\sqrt{n}}.\end{equation} 
Set $K=\sqrt{M}/2$.
Clearly we have $|J(k)|\leq1/2$
for every $k\in\mathbb N$.
Hence, for every $1\leq n\leq M$
and every $a_1\cdots a_n\in \mathbb N^n$ we have $|J(a_1\cdots a_n)|\leq 1/2=K/\sqrt{M}\leq K/\sqrt{n}$ as required.

Let $n\geq M+1$ and $a_1\cdots a_n\in \mathbb N^n$. We may assume $a_1\cdots a_n$ contains $1$, for otherwise
a direct calculation shows $|J(a_1\cdots a_n)|\leq1/(n+1)$. 
Let $N\geq1$ denote the total number of blocks of consecutive $1$s in $a_1\cdots a_n$.
A block of length not exceeding $\sqrt{n}$ is called a short block.
A block which is not short is called a long block.
If $N\geq\sqrt{n}/2$, then Lemma~\ref{3-exp} implies 
$|J(a_1\cdots a_n)|\leq(9/4)^{-\sqrt{n}/2+1}$. This and \eqref{decay1} together yield the desired inequality.

Suppose $N<\sqrt{n}/2$.
If there is no long block, then $\#\{1\leq i\leq n\colon a_i\neq 1\}\geq n-\sqrt{n}N>n/2$. 
Let $j=\min\{i\geq1\colon a_i\neq1\}$ and $k=\max\{i\geq1\colon a_i\neq1\}$.
Define $(\omega_i)_{i\in\mathbb N}\in\Omega$ by
$\omega_i\equiv a_i\mod 2$.
By the mean value theorem and Lemma~\ref{3-exp}, for some $\ell\geq1$ and all $x\in T_\omega^{j-1}(J(a_1\cdots a_n))$ we have
\[\begin{split}1\geq &|T_{\theta^{j}\omega}^{k-j+1}\circ T_\omega^{j-1}(J(a_1\cdots a_n))|\\
=&|T_{\theta^{j}\omega}^{t(\theta^{j}\omega,x)+t(\widehat R(\theta^{j}\omega,x)) +\cdots+ t(\widehat R^{\ell-1}(\theta^{j}\omega,x))}\circ T_\omega^{j-1}(J(a_1\cdots a_n))|\\
\geq&
\left(\frac{9}{4}\right)^{\lfloor \ell/2\rfloor}|T_\omega^{j-1}(J(a_1\cdots a_n))|\geq\left(\frac{9}{4}\right)^{\lfloor \ell/2\rfloor}|J(a_1\cdots a_n)|.\end{split}\]
Since $\ell\geq\lfloor n/2\rfloor-1\geq n/2-2$ we have
$\ell/2\geq n/4-1$, and so 
$\lfloor \ell/2\rfloor\geq\lfloor n/4-1\rfloor=\lfloor n/4\rfloor-1$. Combining this inequality with the above yields 
$|J(a_1\cdots a_n)|\leq (9/4)^{-\lfloor n/4\rfloor+1}$.
By $n\geq M+1\geq 5$ and \eqref{decay1}, we obtain
$(9/4)^{-\lfloor n/4\rfloor+1}\leq 
(9/4)^{-\sqrt{n}/2+1}\leq1/\sqrt{n}$.
If there is a long block, then
there exists $1\leq j\leq n-1$ such that $a_i=1$
for $i=j,\ldots, j+\lfloor \sqrt{n}\rfloor-1$, and thus
 $T_\omega^{j-1}(J(a_1\cdots a_n))\subset J(1^{\lfloor \sqrt{n}\rfloor})\subset[0,1/(\lfloor \sqrt{n}\rfloor+1)$. By the mean value theorem we obtain $|J(a_1\cdots a_n)|\leq 1/ \sqrt{n}$. 
\end{proof}

\subsection{Annealed geometric potential}\label{anneal-sec}
We introduce a function $\varphi\colon(\Omega\times[0,1])\setminus E\to \mathbb R$ by
\[\varphi(\omega,x)=\log p(\omega_1)-\log|T_{\omega_1}'x|,\]
where \[p(\omega_1)=\begin{cases}1-p\ &\text{ if }\omega_1=0,\\
p\ &\text{ if }\omega_1=1.\end{cases}\]
Note that $\varphi$ is unbounded and $\sup\varphi<0.$ 
We call $\varphi$ an {\it annealed geometric potential}.
For $n\in\mathbb N$ write
$S_n\varphi$ for the Birkhoff sum $\sum_{i=0}^{n-1}\varphi\circ R^i$, and put $S_0\varphi\equiv0$ for convenience. The annealed geometric potential ties in  with 
Theorem~\ref{thm-b}.
For all $(\omega,x)\in\Lambda$ and all $n\in\mathbb N$ we have
\[\exp(S_n\varphi(\omega,x))=Q_n^p(\omega)|(T_\omega^n)'x|^{-1}.\]
Compare this formula with \eqref{zpn}.
The next distortion estimate is straight forward.

\begin{lemma}\label{dist-basic}
For all $n\in\mathbb N$, $a_1\cdots a_n\in \mathbb N^n$ and any pair $(\omega,x),(\varrho,y)$ of points in $\varDelta(a_1\cdots a_n)$,
\[S_n\varphi(\omega,x)-S_n\varphi(\varrho,y)\leq 2\sum_{i=1}^n|T^i_\omega x-T^i_\varrho y|.\]
\end{lemma}
\begin{proof}We have \[S_n\varphi(\omega,x)-S_n\varphi(\varrho,y)=\log\frac{|(T_\omega^n)'y|}{|(T_\omega^n)'x|}=\log\frac{|(T_\varrho^n)'y|}{|(T_\varrho^n)'x|}.\]
Then the desired inequality follows from the chain rule and \eqref{ren}.
\end{proof}

 For each $n\in\mathbb N$ define 
\[D_n(\varphi)=\sup \{S_n\varphi(\omega,x)-S_n\varphi(\varrho,y)\colon a_i(\omega,x)=a_i(\varrho,y),\ i=1,\ldots,n\}.\]
Note that $D_1(\varphi)<\infty$, and $D_n(\varphi)$ is decreasing in $n$.
\begin{lemma}\label{mild-lem}
We have
$D_n(\varphi)=O(\sqrt{n})$ $(n\to\infty).$
\end{lemma}
\begin{proof}
Let $n\in\mathbb N$, $a_1\cdots a_n\in \mathbb N^n$ and let
 $(\omega,x),(\varrho,y)\in \varDelta(a_1\cdots a_n)$.
Using Lemma~\ref{dist-basic} and then Lemma~\ref{u-decay}, we have
\[\begin{split}S_n\varphi(\omega,x)-S_n\varphi(\varphi,y)&\leq
2\sum_{i=1}^n
|T_\omega^ix-T_\varrho^iy|\\
&\leq 2+2\sum_{i=1}^{n-1}|J(a_{i+1}\cdots a_n)|\leq K\sum_{i=1}^n\frac{1}{\sqrt{n-i+1}}=O(\sqrt{n}),\end{split}\]
which implies the assertion of the lemma.
\end{proof}


\subsection{Preliminary lemmas for the proof of Theorem~\ref{ldpup-q}}\label{subexp}
One key point in the proof 
of Theorem~\ref{ldpup-q} is that
  the measure 
$\int_\Omega\tilde\mu^\omega_n(\cdot)dm_p(\omega)$ equals $\tilde\mu_n(\cdot)$ up to subexponential factors. 
To show this, we first provide subexponential bounds on the normalizing constants
$Z_{\omega,n}$ in \eqref{zomegan}.

 \begin{lemma}\label{b-d}
 For all $\omega\in\Omega$ and $n\in\mathbb N$ we have
\[\exp(-D_n(\varphi) )\leq Z_{\omega,n}\leq \exp(D_n(\varphi)).\]
In particular, $Z_{p,n}$ is finite for all $p\in(0,1)$ and all $n\in\mathbb N$.
\end{lemma}
\begin{proof}
Let $\omega\in\Omega$, $n\in\mathbb N$ and let $a_1\cdots a_n\in\mathbb N^{\mathbb N}$ satisfy
 $\omega_i\equiv a_i$ mod $2$ for $i=1,\ldots,n$. Clearly,
 $J(a_1\cdots a_n)\cap {\rm Fix}(T_\omega^n)$ is a singleton.
  Let $x(a_1\cdots a_n)$
denote the element of this singleton.  
  By the mean value theorem,
  for each $a_1\cdots a_n\in \mathbb N^n$ there exists $y(a_1\cdots a_n)\in J(a_1\cdots a_n)$ such that
  $|(T_\omega^n)'y(a_1\cdots a_n)|^{-1}=|J(a_1\cdots a_n)|.$ We have
  \[\exp(-D_n(\varphi))|J(a_1\cdots a_n)|\leq|(T_\omega^n)'x(a_1\cdots a_n)|^{-1}\leq \exp(D_n(\varphi))|J(a_1\cdots a_n)|.\]
Summing the first inequality over all relevant $a_1\cdots a_n$ gives 
\[Z_{\omega,n}\geq \exp(-D_n(\varphi))\sum_{\substack{a_1\cdots a_n\in \mathbb N^n \\ a_i\equiv \omega_i\mod 2\\ \ i=1,\ldots, n   }}|J(a_1\cdots a_n)|
=\exp(-D_n(\varphi)),\]
as required.
Summing the second inequality in the double inequalities over all relevant $a_1\cdots a_n$ yields the required upper bound.
\end{proof}

\begin{lemma}\label{lem-cal1}
For any Borel subset $\mathcal C$ of $\mathcal M(\Lambda)$ and every $n\in\mathbb N$,
\[ \exp(-2D_n(\varphi))\tilde\mu_n(\mathcal C)\leq\int_{\Omega}\tilde\mu_n^\omega(\mathcal C)dm_p(\omega)
\leq \exp(2D_n(\varphi))\tilde\mu_n(\mathcal C).\]
\end{lemma}
\begin{proof}
By Lemma~\ref{b-d}, for all $\omega\in\Omega$ and all $n\in\mathbb N$ we have
\begin{equation}\label{inequa1}\exp(-2D_n(\varphi))\leq {Z_{\omega,n} \Big/ \int_\Omega
Z_{\omega',n} dm_p(\omega')}\leq \exp(2D_n(\varphi)).\end{equation}
By the definitions of $\tilde\mu_n$ and $\tilde\mu_n^\omega$,
for any Borel subset $\mathcal C$ of $\mathcal M(\Lambda)$ and all $n\in\mathbb N$,
\begin{equation}\label{inequa2}\begin{split}
\tilde\mu_n(\mathcal C)&=\frac{1}{Z_{p,n}}\sum_{
\substack{(\omega,x)\in{\rm Fix}(R^n)\\ V_n^R(\omega,x)\in\mathcal C}}Q_p^n(\omega) |(T_\omega^n)'x|^{-1}\\
&=\int_\Omega \sum_{\substack{x\in{\rm Fix}(T^n_\omega)\\ V_n^R(\omega,x)\in\mathcal C}}|(T_{\omega}^n)'x|^{-1} dm_p(\omega) \Big/\int_\Omega Z_{\omega',n}  dm_p(\omega') \\
&=\int_\Omega \tilde\mu_n^\omega(\mathcal C)\left({Z_{\omega,n} \Big/ \int_\Omega
Z_{\omega',n} dm_p(\omega')}\right)dm_p(\omega).
\end{split}\end{equation}
Combining \eqref{inequa1} and \eqref{inequa2} yields the desired inequality. \end{proof}



The next lemma gives an upper bound for each closed subset of $\mathcal M(\Lambda)$ by the rate function $I_p$, but is not sufficient for Theorem~\ref{ldpup-q}(a) since the set of permissible samples depends on the closed set in consideration.
\begin{lemma}\label{lem-cal2}For any closed subset $\mathcal C$ of $\mathcal M(\Lambda)$, there exists a Borel subset $\Gamma(\mathcal C)$ of $\Omega$ such that $m_p(\Gamma(\mathcal C))=1$ and
for every $\omega\in\Gamma(\mathcal C)$, \[\limsup_{n\to\infty}\frac{1}{n}\log \tilde\mu_{n}^\omega(\mathcal C)\leq -\inf_{\mathcal C} I_p. \]\end{lemma}
\begin{proof} Let $\mathcal C$ be a closed subset of $\mathcal M(\Lambda)$. We may assume $\inf_{\mathcal C}I_p>0$, for otherwise the inequality is obvious. 
We first consider the case $\inf_{\mathcal C}I_p<\infty$.
For $\varepsilon\in(0,1)$ and $n\geq1$, set 
\[\Omega_{\varepsilon,n}=\left\{\omega\in\Omega\colon \tilde\mu_n^\omega(\mathcal C)\geq \exp\left(-n(1-\varepsilon)\inf_{\mathcal C} I_p\right)\right\}.\]
By Markov's inequality and the second inequality in Lemma~\ref{lem-cal1}, 
\[
\begin{split}m_p(\Omega_{\varepsilon,n})
&\leq \exp\left(n(1-\varepsilon)\inf_{\mathcal C} I_p\right)\int_{\Omega}\tilde\mu_n^\omega(\mathcal C)d m_p(\omega)\\
&\leq \exp(2D_n(\varphi))\exp\left(n(1-\varepsilon)\inf_{\mathcal C}I_p\right)\tilde\mu_n(\mathcal C).\end{split}\]
By the LDP in Theorem~\ref{level-2-thm}(a),
$m_p(\Omega_{\varepsilon,n})$ decays exponentially as $n$ increases.
By Borel-Cantelli's lemma, the inequality $\tilde\mu_n^\omega(\mathcal C)\geq \exp(-n(1-\varepsilon)\inf_{\mathcal C} I_p )$ holds only for finitely many $n$ for $m_p$-almost every $\omega\in\Omega.$
Since $\varepsilon\in(0,1)$ is arbitrary, we obtain 
the desired inequality for $m_p$-almost every $\omega\in\Omega$. 

To treat the remaining case $\inf_{\mathcal C}I_p=\infty$, for $k,n\in\mathbb N$ we set 
\[\Omega_{k,n}=\left\{\omega\in\Omega\colon \tilde\mu_n^\omega(\mathcal C)\geq e^{-kn }\right\}.\]
By Markov's inequality and Lemma~\ref{lem-cal1},
\[m_p(\Omega_{k,n})
\leq e^{kn}\int_{\Omega}\tilde\mu_n^\omega(\mathcal C)d m_p(\omega)\leq \exp(2D_n(\varphi))e^{kn}\tilde\mu_n(\mathcal C).\]
Since $\mathcal C$ is closed, the LDP in Theorem~\ref{level-2-thm}(a) gives $\limsup_n(1/n)\log \tilde\mu_n(\mathcal C)\leq-\inf_{\mathcal C}I_p=-\infty$. Hence
$m_p(\Omega_{k,n})$ decays exponentially as $n$ increases.
By Borel-Cantelli's lemma, there exists a Borel subset $\Gamma_k(\mathcal C)$ of $\Omega$ such that $m_p(\Gamma_k(\mathcal C))=1$, and for any $\omega\in\Gamma_k(\mathcal C)$ the inequality $\tilde\mu_n^\omega(\mathcal C)\geq e^{-kn}$ holds only for finitely many $n$. 
Put $\Gamma(\mathcal C)=\bigcap_{k=1}^\infty\Gamma_k(\mathcal C)$.
We have $m_p(\Gamma(\mathcal C))=1$, and  $\limsup_n(1/n)\log \tilde\mu_{n}^\omega(\mathcal C)=-\infty=-\inf_{\mathcal C}I_p$ for all $\omega\in\Gamma(\mathcal C)$ as required.
\end{proof}

Since $\mathcal M(\Lambda)$ is non-compact, we need the following auxiliary lemma that leads to the exponential tightness of $(\tilde\mu^{\omega}_n)_{n=1}^\infty$ as in Proposition~\ref{ldpup-q}(a).
\begin{lemma}\label{lem-cal3}For any $L>0$ there exists a compact subset $\mathcal K_L$ of $\mathcal M(\Lambda)$ and a Borel subset $\Gamma_L$ of $\Omega$ such that $m_p(\Gamma_L)=1$ and 
for every $\omega\in\Gamma_L$,
\[\limsup_{n\to\infty}\frac{1}{n}\log\tilde\mu^\omega_n(\mathcal M(\Lambda)\setminus \mathcal K_L)\leq-L.\]\end{lemma}
\begin{proof}
By the exponential tightness of
$(\tilde\mu_n)_{n=1}^\infty$ in Theorem~\ref{level-2-thm}(a),  
for any $L>0$
there is a compact subset $\mathcal K_L$ of
$\mathcal M(\Lambda)$ such that 
\begin{equation}\label{exp-sup}\limsup_{n\to\infty}\frac{1}{n}\log\tilde\mu_n(\mathcal M(\Lambda)\setminus \mathcal K_L)\leq-2L.\end{equation}
For $n\in\mathbb N$, set 
\[\Omega_{L,n}=\left\{\omega\in\Omega\colon \tilde\mu_n^\omega(\mathcal M(\Lambda)\setminus \mathcal K_L)\geq e^{-Ln} \right\}.\]
By Markov's inequality and  Lemma~\ref{lem-cal1},
\[m_p(\Omega_{L,n})
\leq e^{Ln}\int_{\Omega}\tilde\mu_n^\omega(\mathcal M(\Lambda)\setminus \mathcal K_L)d m_p(\omega)\\
\leq \exp(2D_n(\varphi))e^{Ln}\tilde\mu_n(\mathcal M(\Lambda)\setminus \mathcal K_L).\]
By Lemma~\ref{mild-lem} and \eqref{exp-sup},
$m_p(\Omega_{L,n})$ decays exponentially as $n$ increases.
By Borel-Cantelli's lemma, the number of those $n\in\mathbb N$
with $\tilde\mu_n^\omega(\mathcal M(\Lambda)\setminus \mathcal K_L)\geq e^{-Ln}$ is finite for $m_p$-almost every $\omega\in\Omega.$
\end{proof}

\subsection{Proof of Theorem~\ref{ldpup-q}}\label{pf-sample}
We fix a metric on $\mathcal M(\Lambda)$ that generates the weak* topology, and a
  countable dense subset $\mathcal D$ of on $\mathcal M(\Lambda)$.
 For $\mu\in \mathcal D$, $L\in\mathbb N$ 
let $B(\mu,1/L)$ denote the closed ball of radius $1/L$ about $\mu$.
 By Lemma~\ref{lem-cal2}, there exists a Borel subset $\Gamma(B(\mu,1/L))$ of $\Omega$ with full $m_p$-measure such that
if $\omega\in\Gamma(B(\mu,1/L))$ then
\begin{equation}\label{eq-zero}\limsup_{n\to\infty}\frac{1}{n}\log \tilde\mu_{n}^\omega( B(\mu,1/L))\leq -\inf_{B(\mu,1/L)}I_p.\end{equation}
In view of Lemma~\ref{lem-cal3}, we fix an increasing sequence $(\mathcal K_L)_{L=1}^\infty$ of compact subsets of $\mathcal M(\Lambda)$ 
and a sequence $(\Gamma_L)_{L=1}^\infty$ of Borel subsets of $\Omega$ with full $m_p$-measure such that $\bigcup_{L=1}^\infty\mathcal K_L=\mathcal M(\Lambda)$,  and for all  $L\in\mathbb N$ and all
$\omega\in\Gamma_L$,
\begin{equation}\label{eq-un}\limsup_{n\to\infty}\frac{1}{n}\log\tilde\mu^\omega_n(\mathcal M(\Lambda)\setminus \mathcal K_L)\leq-L.\end{equation}
We set \[\Gamma=\left(\bigcap_{\mu\in\mathcal D}\bigcap_{L=1}^\infty\Gamma(B(\mu,1/L))\right)\cap \left(\bigcap_{L=1}^\infty\Gamma_L\right).\]
Clearly we have $m_p(\Gamma)=1$.
If $\omega\in\Gamma$, then $(\tilde\mu_n^\omega)_{n=1}^\infty$ is exponentially tight by \eqref{eq-un}.

Let $\mathcal C$ be a non-empty closed subset of $\mathcal M(\Lambda)$ and let $L\in\mathbb N$.
Let $\mathcal G$ be an open subset of $\mathcal M(\Lambda)$ that contains $\mathcal C\cap \mathcal K_L$.
Since $\mathcal C\cap \mathcal K_L$ is compact,
there exists a finite subset $\{\mu_1,\ldots,\mu_s\}$ of $\mathcal D$ and $L_1,\ldots,L_s\in\mathbb N$ such that
$\mathcal C\cap \mathcal K_L\subset \bigcup_{i=1}^sB(\mu_i,1/L_i)\subset\mathcal G$.
By \eqref{eq-zero} applied to each of these closed balls, we have
\[\begin{split}\limsup_{n\to\infty}\frac{1}{n}\log \tilde\mu_{n}^\omega(\mathcal C\cap \mathcal K_L)&\leq\max_{1\leq i\leq s}\limsup_{n\to\infty}\frac{1}{n}\log \tilde\mu_n^{\omega}(B(\mu_i,1/L_i))\\
&\leq\max_{1\leq i\leq s}\left( -\inf_{B(\mu_i,1/L_i)}I_p\right)\leq-\inf_{\mathcal G} I_p.\end{split}\]
Since $\mathcal G$ is an arbitrary open set containing $\mathcal C\cap\mathcal K_L$ and $I_p$ is lower semicontinuous,
\begin{equation}\label{eq-deux}\limsup_{n\to\infty}\frac{1}{n}\log \tilde\mu_{n}^\omega(\mathcal C\cap \mathcal K_L)\leq-\inf_{\mathcal C\cap \mathcal K_L} I_p.\end{equation}
From \eqref{eq-un} and \eqref{eq-deux}, for every $\omega\in\Gamma$ we obtain
\begin{equation}\label{eq-trois}\limsup_{n\to\infty}\frac{1}{n}\log\tilde\mu^\omega_n(\mathcal C)\leq\max\left\{-\inf_{\mathcal C\cap \mathcal K_L} I_p,-L\right\}.\end{equation}
If $L\geq\inf_{\mathcal C\cap\mathcal K_L}I_p$, then \eqref{eq-trois} yields \[\limsup_{n\to\infty}\frac{1}{n}\log\tilde\mu^\omega_n(\mathcal C)\leq-\inf_{\mathcal C\cap \mathcal K_L} I_p\leq -\inf_{\mathcal C}I_p.\]
Combining this with \eqref{eq-un} we obtain the desired inequality. 
If $L<\inf_{\mathcal C\cap\mathcal K_L}I_p$ for all $L\in\mathbb N$, then 
we obtain $\inf_\mathcal CI_p=\infty$ since
$(\mathcal K_L)_{L=1}^\infty$ is increasing and $\bigcup_{L=1}^\infty\mathcal K_L=\mathcal M(\Lambda)$.
 Moreover, \eqref{eq-trois} yields $\limsup_n(1/n)\log\tilde\mu^\omega_n(\mathcal C)=-\infty.$ The proof of
Theorem~\ref{ldpup-q}(a) is complete.

By Theorem~\ref{ldpup-q}(a), 
$(\tilde\mu_n^\omega)_{n=1}^\infty$ is tight for $m_p$-almost every $\omega\in\Omega$. By Prohorov's theorem, it has a limit point.
Let $(\tilde{\mu}_{n_j}^\omega)_{j=1}^\infty$
be an arbitrary convergent subsequence of $(\tilde\mu_n^\omega)_{n=1}^\infty$ with the limit measure $\tilde\mu^\omega$. For a proof of
Theorem~\ref{ldpup-q}(b)
it suffices to show 
 $\tilde\mu^\omega=\delta_{m_p\otimes\lambda_p}$.

We fix a metric that generates the weak* topology on $\mathcal M(\Lambda)$. 
Since $I_p$ is a good rate function by Theorem~\ref{level-2-thm}(a), for any $c>0$ the level set
 $I_p^c=\{\mu\in\mathcal M(\Lambda)\colon I_p(\mu)\leq c\}$
is compact.  
Let $\nu\in\mathcal M(\Lambda)\setminus\{m_p\otimes\lambda_p\}$. 
By the last assertion of Proposition~\ref{CP} we have $I_p(\nu)>0$, and so $\nu\notin I_p^{I(\nu)/2}$.
 Take $r>0$ such that the closed ball $B(\nu,r)$ of radius $r$ about $\nu$ in $\mathcal M(\Lambda)$
 does not intersect $I_p^{I(\nu)/2}$.
By the weak* convergence of
$(\tilde{\mu}_{n_j}^\omega)_{j=1}^\infty$ to $\tilde\mu^\omega$
 and the large deviations upper bound for closed sets in Theorem~\ref{ldpup-q}(a),
we have
\[\begin{split}\tilde\mu^\omega({\rm int}( B(\nu,r)))&\leq\liminf_{j\to\infty}\tilde\mu_{n_j}^\omega({\rm int}(B(\nu,r)))
\leq\limsup_{j\to\infty}\tilde\mu_{n_j}^\omega( B(\nu,r))\\&\leq\limsup_{j\to\infty}\exp(-I_p(\nu) n_j/2)=0.\end{split}\]
Hence, the support of $\tilde\mu^\omega$ does not contain $\nu$. Since
 $\nu$ is an arbitrary element of
 $\mathcal M(\Lambda)$ which is not $m_p\otimes\lambda_p$, it follows that
 $\tilde\mu^\omega=\delta_{m_p\otimes\lambda_p}$.
 The proof of Theorem~\ref{ldpup-q}(b) is complete.
 \qed

\begin{remark}
Since $\mathcal M(\Lambda)$ is non-compact,
the tightness in Theorem~\ref{ldpup-q}(a) was used in establishing the convergence in Theorem~\ref{ldpup-q}(b).
Nevertheless, $\mathcal M(\Omega\times[0,1])$ is compact.
By applying the Contraction Principle to the inclusion $\mathcal M(\Lambda)\hookrightarrow\mathcal M(\Omega\times[0,1])$, one can transfer
the LDP in Theorem~\ref{level-2-thm}(a) to the LDP for 
the sequence $(\tilde\mu_n)_{n=1}^\infty$ viewed as a sequence in $\mathcal M^2(\Omega\times[0,1])$. Using the latter LDP, 
one can establish a version of the upper bound in Theorem~\ref{ldpup-q}(a) for any closed subset of $\mathcal M(\Omega\times[0,1])$,
as well as the convergence of $(\tilde\mu_n)_{n=1}^\infty$ to $\delta_{m_p\otimes\lambda_p}$ in   $\mathcal M^2(\Omega\times [0,1])$. 
These are actually sufficient for the proof of Theorem~\ref{thm-a}.

One merit of considering large deviations on the non-compact space $\mathcal M(\Lambda)$ rather than on $\mathcal M(\Omega\times[0,1])$ is that one can permit bounded continuous functions on $\Lambda$ that are naturally associated with the random continued fraction expansion \eqref{r-expansion},  and do not have continuous extensions  to $\Omega\times[0,1]$. See Corollary~\ref{lyacor2} for details. \end{remark}

 \section{Establishing the LDP for the Gauss-R\'enyi map}
This last section is mostly dedicated to the proof of Theorem~\ref{level-2-thm}.
In \S\ref{general} we summarize results on the thermodynamic formalism for the countable full shift. In \S\ref{ind-sec} we 
consider an inducing scheme of the full shift and introduce a symbolic coding of the associated induced system.
In \S\ref{LDP-shift} we recall the result of the second-named author \cite{Tak22} that give a sufficient condition for the level-2 LDP on periodic points in terms of induced potentials. We also recall the result in \cite{Tak20} on the uniqueness of minimizer of the rate function.
In order to implement all these results,
 in $\S$\ref{coding} 
 we show that the Gauss-R\'enyi map 
is topologically conjugate to the shift map 
on the countable full shift.
In \S\ref{random-dist} we perform distortion estimates for an induced version of the annealed geometric potential $\varphi$.  In \S\ref{variational-sec} we establish the existence and uniqueness of the equilibrium state for the symbolic version of the potential $\varphi$, and show that this equilibrium state is the symbolic version of the measure $m_p\otimes\lambda_p$.
In \S\ref{section4-1} we complete the proof of Theorem~\ref{level-2-thm}. In \S\ref{level-1-sec} we state two corollaries of independent interest on annealed and quenched level-1 large deviations, and apply them to the problem of frequency of digits in the random continued fraction expansion.

\subsection{Thermodynamic formalism for the countable full shift}\label{general}
Consider the {\it countable full shift} \begin{equation}\label{CMS}\mathbb N^{\mathbb N}=\{z=(z_n)_{n=1}^\infty\colon z_n\in\mathbb N\text{ for }n\in\mathbb N\},\end{equation} which is the cartesian product topological space of the discrete space $\mathbb N$.
We introduce main constituent components of the thermodynamic formalism for the countable full shift \eqref{CMS}, and
state a variational principle and a relationship between equilibrium states and Gibbs states.
  Our main reference is \cite{MauUrb03} that contains results on countable Markov shifts which are not necessarily the full shift.

The left shift $\sigma\colon\mathbb N^{\mathbb N}\to\mathbb N^{\mathbb N}$ given by $\sigma (z_n)_{n=1}^\infty=(z_{n+1})_{n=1}^\infty$ is continuous.
For $n\in\mathbb N$ and $a_1\cdots a_n\in\mathbb N^n$, define an {\it $n$-cylinder}
\[[a_1\cdots a_n]=\{z\in \mathbb N^{\mathbb N}\colon z_i=a_i\text{ for }
 i=1,\ldots,n\}.\]
Let $\mathcal M(\mathbb N^{\mathbb N},\sigma)$ denote the set of $\sigma$-invariant Borel probability measures. For each $\mu\in\mathcal M(\mathbb N^{\mathbb N},\sigma)$, let $h(\mu)\in[0,\infty]$ denote the measure-theoretic entropy of $\mu$ with respect to $\sigma$.
 Let $\phi\colon \mathbb N^{\mathbb N}\to\mathbb R$ be a function, called a {\it potential}. 
 For each $n\in\mathbb N$ we write
$S_n\phi$ for the Birkhoff sum 
$\sum_{i=0}^{n-1}\phi\circ \sigma^i$,
and
 introduce a {\it pressure} 
\[P(\phi)=\lim_{n\to\infty}\frac{1}{n}\log\sum_{a_1\cdots  a_n\in \mathbb N^n}
\sup_{[a_1\cdots  a_n]}\exp S_n\phi. \]
This limit exists by the sub-additivity, 
which is never $-\infty$.
We say:
 \begin{itemize}\item $\phi$ is {\it acceptable} if it is uniformly continuous and satisfies
\[\sup_{a\in \mathbb N}\left(\sup_{[a]}\phi-\inf_{[a]}\phi\right)<\infty;\]
\item $\phi$ is {\it locally H\"older continuous} 
if there exist constants $K>0$ and $\gamma\in(0,1)$ such that ${\rm var}_n(\phi)\leq K\gamma^n$, where
\[{\rm var}_n(\phi)=\sup\{\phi(z)-\phi(w)\colon z,w\in\mathbb N^{\mathbb N},\  z_i=w_i\ \text{ for } i=1,\ldots,n\}.\]
\end{itemize}

  Let $\phi\colon \mathbb N^{\mathbb N}\to\mathbb R$ be acceptable and satisfy $P(\phi)<\infty$. Then
 $\sup\phi$ is finite (see \cite[Proposition~2.1.9]{MauUrb03}). 
 Let \[\mathcal M_\phi(\mathbb N^{\mathbb N},\sigma)=\left\{\mu\in\mathcal M(\mathbb N^{\mathbb N},\sigma)\colon\int\phi d\mu>-\infty\right\}.\]
 By \cite[Theorem~2.1.7]{MauUrb03}, for any $\mu\in \mathcal M_\phi(\mathbb N^{\mathbb N},\sigma)$ we have 
 $h(\mu)+\int\phi d\mu\leq P(\phi)<\infty$,
 and so $h(\mu)<\infty$. 
   The following equality is known as the variational principle.
\begin{prop}[{\cite[Theorem~2.1.7, Theorem~2.1.8]{MauUrb03}}]\label{thermo1}Let $\phi\colon\mathbb N^{\mathbb N}\to\mathbb R$ be acceptable and satisfy $P(\phi)<\infty$. Then
\[P(\phi)=\sup\left\{h(\mu)+\int\phi  d\mu\colon\mu\in\mathcal M_\phi(\mathbb N^{\mathbb N},\sigma)\right\}.\]\end{prop}
Let $\phi\colon\mathbb N^{\mathbb N}\to\mathbb R$ be acceptable and satisfy $P(\phi)<\infty$. 
A measure $\mu\in\mathcal M_\phi(\mathbb N^{\mathbb N},\sigma)$ is called {\it an equilibrium state for the potential $\phi$}
if \[P(\phi)=h(\mu)+\int\phi d\mu.\]
 A measure $\mu\in\mathcal M(\mathbb N^{\mathbb N})$ is called {\it a Gibbs state for the potential $\phi$} if there exists a constant $K\geq1$ 
such that 
for all $n\in\mathbb N$, all $a_1\cdots a_n\in\mathbb N^n$
 and all $x\in [a_1\cdots a_n]$,
\[
K^{-1}\leq\frac{\mu([a_1\cdots a_n])}{\exp (S_n\phi(x )-P(\phi)n)}\leq K.
\]
\begin{prop}[{\cite[Theorem~2.2.9,\ Corollary~2.7.5]{MauUrb03}}]\label{thermo2} 
Let $\phi\colon\mathbb N^{\mathbb N}\to\mathbb R$ 
be locally H\"older continuous and satisfy $P(\phi)<\infty$. 
 Then there exists a unique shift-invariant Gibbs state $\mu_{\phi}$ for $\phi$. If $\int\phi d\mu_\phi>-\infty$, then $\mu_\phi$ is the unique equilibrium state for $\phi$. \end{prop}

 \subsection{Coding of the induced system}\label{ind-sec}

Consider the inducing scheme $(\mathbb N^{\mathbb N}\setminus[1],t_{\mathbb N^{\mathbb N}\setminus[1]})$
of the left shift
 $\sigma\colon\mathbb N^{\mathbb N}\to\mathbb N^{\mathbb N}$. We show that the associated induced system $\widehat\sigma\colon{\widehat{\mathbb N}}^{\mathbb N}\to{\widehat{\mathbb N}}^{\mathbb N}$
is in a natural way topologically conjugate to the full shift over an infinite alphabet. 

We introduce the empty word $\emptyset$ by the rule
 $\omega\emptyset=\omega=\emptyset \omega$ 
 for any word $\omega$ from $\mathbb N$.
For each $n\in\mathbb N$, write
$1^{n}$ for $11\cdots 1\in\mathbb N^n$, the $n$-string of $1$. We set 
$1^0=\emptyset$ for convenience. 
 We introduce an infinite alphabet
 \begin{equation}\label{ahat}\mathbb M=\left\{\bigcup_{b\in \mathbb N\setminus\{1\} }[a 1^n b]\colon a\in \mathbb N\setminus\{1\}\text{ and }n\in\mathbb N\cup\{0\}\right\},\end{equation}
 which is a collection of pairwise disjoint subsets of $\mathbb N^{\mathbb N}\setminus[1]$.
 We endow $\mathbb M$ with the discrete topology, and
introduce the countable full shift \begin{equation}\label{CMS'}{\mathbb M}^{\mathbb N}=\{(x_n)_{n=1}^\infty\colon x_n\in\mathbb M\text{ for   }n\in\mathbb N\},\end{equation}
which is the cartesian product topological space of $\mathbb M$. 
Clearly $\mathbb M^{\mathbb N}$ is topologically isomorphic to $\mathbb N^{\mathbb N}$. With a slight abuse of notation 
let $\sigma\colon{\mathbb M}^{\mathbb N}\to{\mathbb M}^{\mathbb N}$ denote the left shift. 

We define a map
$\iota\colon\mathbb M^{\mathbb N}\to{\widehat{\mathbb N}}^{\mathbb N}$ as follows. Let $(x_n)_{n=1}^\infty\in\mathbb M^{\mathbb N}$.
 By the definition of $\mathbb M$ in \eqref{ahat}, 
for every $n\in\mathbb N$ we have
    $x_n=\bigcup_{b\in \mathbb N\setminus\{1\}}[a_n1^{j_n} b]$
where $a_n\in\mathbb N\setminus\{1\}$ and $j_n\in\mathbb N\cup\{0\}$. 
We set 
   \[ \iota( (x_n)_{n=1}^\infty)\in\bigcap_{n=1}^\infty[a_11^{j_1}a_21^{j_2}\cdots a_n1^{j_n}].\]

    \begin{lemma}\label{represent-prop}The map $\iota$ is a homeomorphism, and satisfies 
    $\iota\circ\sigma=\widehat\sigma\circ\iota$.
    \end{lemma}
    \begin{proof} Clearly $\iota$ is continuous and injective. For every $a\in \mathbb N\setminus\{1\}$ and every $n\in \mathbb N\cup\{0\}$, the set $\bigcup_{b\in \mathbb N\setminus\{1\} }[a 1^n b]$ 
is mapped by $\widehat\sigma$ 
bijectively onto $\mathbb N^{\mathbb N}\setminus[1]$. Moreover, the collection of sets of this form defines a partition of the set 
$\bigcup_{k=1}^\infty\{t=k\}$,
namely    \[\bigcup_{k=1}^\infty\{t=k\}=\bigcup_{a\in\mathbb N\setminus\{1\}}\bigcup_{n\in\mathbb N\cup\{0\}}\bigcup_{b\in\mathbb N\setminus\{1\} }[a1^nb].\]
   All the unions are disjoint unions. It follows that $\iota(\mathbb M^{\mathbb N})=\widehat{\mathbb N}^{\mathbb N}$. The last assertion follows from the definition of $\iota.$ \end{proof}


\subsection{Level-2 LDP for the countable full shift}\label{LDP-shift}
Let $\phi\colon \mathbb N^{\mathbb N}\to\mathbb R$
be acceptable and satisfy $P(\phi)<\infty$.
We are concerned with the LDP a sequence $(\tilde\nu_n)_{n=1}^\infty$ of Borel probability measures on $\mathcal M(\mathbb N^{\mathbb N})$ given by
\begin{equation}\label{nu_n}\tilde\nu_n=\frac{1}{Z_n(\phi)}\sum_{ x\in{\rm Fix}(\sigma^n) } \exp (S_n\phi(x ))\delta_{V_n^\sigma(x) },\end{equation}
where 
$V_n^\sigma(x)\in\mathcal M(\mathbb N^{\mathbb N})$ denotes the uniform probability distribution on the orbit $(\sigma^ix)_{i=0}^{n-1}$, and $\delta_{V_n^\sigma(x) }$ denotes the Borel probability measure on $\mathcal M(\mathbb N^{\mathbb N})$ that is the unit point mass at $V_n^\sigma(x)$, and 
$Z_n(\phi)$ denotes the normalizing constant.
We introduce  
 a free energy $ F_\phi \colon \mathcal M(\mathbb N^{\mathbb N}) \to [-\infty, 0]$ by
\[ F_\phi(\mu)
=
\begin{cases}h(\mu)+\int\phi d\mu &\text{ if $\mu\in\mathcal M_\phi(\mathbb N^{\mathbb N},\sigma)$},\\
-\infty&\text{ otherwise.}\end{cases}\]
The function $-F_\phi+P(\phi)$ is a natural candidate for the rate function of this LDP.
However, this function may not be lower semicontinuous since the entropy function is not upper semicontinuous.
Hence, we take the lower semicontinuous regularization of $-F_\phi+P(\phi)$.
Define $I_{\phi}\colon\mathcal M(\mathbb N^{\mathbb N})\to[0,\infty]$ by
\begin{equation}\label{rate-def}
I_{\phi}(\mu)
=
-\inf_{\mathcal G \ni \mu}\sup_{\nu\in\mathcal G}F_\phi(\nu)+P(\phi),
\end{equation}
where the supremum is taken over all measures in an open subset $\mathcal G$ of $\mathcal M(\mathbb N^{\mathbb N})$ that contains $\mu$, and the infimum is taken over all such open subsets.
Then $I_\phi$ is lower semicontinuous and satisfies $I_\phi\leq -F_\phi+P(\phi)$.

If there is a Gibbs state for the potential $\phi$, then the LDP holds for $(\tilde\nu_n)_{n=1}^\infty$ from the result in \cite{Tak19}. Due to the existence of the neutral fixed point of the R\'enyi map $T_1$, the annealed Gauss-R\'enyi measure $\eta_p$ is not a Gibbs state for the potential $\psi$ (see Lemma~\ref{gibb}). Hence \cite{Tak19} cannot be applied to $(\mathbb N^{\mathbb N},\psi)$. Instead we apply the result in 
\cite{Tak22} on the LDP for $(\tilde\nu_n)_{n=1}^\infty$ when a Gibbs state for $\phi$ does not exist.


 Using the conjugacy $\iota$ in \S\ref{ind-sec}, we introduce
a parametrized family of {\it twisted induced potentials} 
 $\Phi_{\gamma}\colon {\mathbb M}^{\mathbb N}\to\mathbb R$ ($\gamma\in\mathbb R$) by 
 \begin{equation}\label{ind-po}\Phi_{\gamma}(\iota(x))=
 S_{t_{\mathbb N^\mathbb N\setminus[1]}(\iota(x))}\phi(\iota(x))-\gamma t_{\mathbb N^\mathbb N\setminus[1]}(\iota(x)).\end{equation} 

\begin{thm}[{\cite[Theorem~A]{Tak22}}]\label{LDP-per}Let $\phi\colon\mathbb N^{\mathbb N}\to\mathbb R$ be acceptable and satisfy $P(\phi)<\infty$. 
Suppose the twisted induced potentials $\Phi_{\gamma}\colon \mathbb M^{\mathbb N}\to\mathbb R$ $(\gamma\in\mathbb R)$ 
are locally H\"older continuous, and there exists $\gamma_0\in\mathbb R$ such that $P(\Phi_{\gamma_0})=0$. Then $(\tilde\nu_n)_{n=1}^\infty$  is exponentially tight and satisfies the LDP with the good rate function $I_\phi$.
\end{thm}

The uniqueness of minimizer of the rate function $I_\phi$ does not follow from Theorem~\ref{LDP-per} and should be examined on a case-by-case basis. An ideal situation is that the shift-invariant Gibbs state for $\phi$ is unique, the equilibrium state for $\phi$ is unique, the minimizer of $I_\phi$ is unique, and all these three coincide. However this is not always the case.
Under the hypothesis of Theorem~\ref{LDP-per}, by virtue of Proposition~\ref{thermo2} there exists a unique Gibbs state for the potential $\phi$. If moreover $\phi$ is integrable against the Gibbs state, then it is the unique equilibrium state for $\phi$, and clearly is a minimizer of $I_\phi$. Conversely, a minimizer of $I_\phi$ may not be an equilibrium state for $\phi$ in general: 
an example of a potential $\phi\colon\mathbb N^{\mathbb N}\to\mathbb R$ can be found in  \cite{Sar03} for which there is a Gibbs state $\mu\in\mathcal M(\mathbb N^{\mathbb N},\sigma)$ such that $I_\phi(\mu)=0$ and  $\mu$ is not an equilibrium state since $\int\phi d\mu=-\infty$.

Under additional hypothesis on the potential, one can show that 
any minimizer is an equilibrium state.
We say $\phi\colon\mathbb N^{\mathbb N}\to\mathbb R$ is {\it summable} if
$\sum_{k\in\mathbb N}\sup_{[k]}e^{\phi}$ is finite. 
If $\phi$ is summable, then $P(\phi)<\infty$.
Set \[\beta_\infty(\phi)=\inf\left\{\beta\in\mathbb R\colon \text{$\beta\phi$ is summable}\right\}.\]
\begin{prop}\label{unique-prop}
Let $\phi\colon \mathbb N^{\mathbb N}\to\mathbb R$ be uniformly continuous and
summable with $\beta_\infty(\phi)<1$. Then, any minimizer of $I_\phi$ is an equilibrium state for the potential $\phi$. \end{prop}
A proof of this proposition is briefly outline as follows.
By the definition \eqref{rate-def}, if $\mu$ is a minimizer of $I_\phi$ then 
there is a sequence $(\mu_k)_{k=1}^\infty$ in $\mathcal M_\phi(\mathbb N^{\mathbb N},\sigma)$ that converges to $\mu$ in the weak* topology with
$\lim_k F_\phi(\mu_k)=0$. 
Based on this information we show that $\mu$ is an equilibrium state for $\phi$.
The case $\lim_kh(\mu_k)=0$ is easy to handle, while the case $\lim_kh(\mu_k)=\infty$ (and hence $\lim_k\int\phi d\mu_k\to-\infty$)
requires attention. A key ingredient in the latter case is the upper semicontinuity of the map
$\mu_k\mapsto h(\mu_k)/(-\int\phi d\mu_k)$, as proved in \cite[Theorem~2.4]{Tak20} inspired by \cite[Lemma~6.5]{FJLR15}. 
\begin{proof}[Proof of Proposition~\ref{unique-prop}] The following proof is almost a repetition of the proof of \cite[Theorem~2.1]{Tak20} for the reader's convenience.   Considering $\phi-P(\phi)$ instead of $\phi$, we may assume $P(\phi)=0$. 
Let $\mu\in\mathcal M(\mathbb N^{\mathbb N},\sigma)$ be a minimizer of $I_\phi$. Since 
$\mathcal M(\mathbb N^{\mathbb N},\sigma)$ is a closed subset of
$\mathcal M(\mathbb N^{\mathbb N},\sigma)$, 
$\mu$ is shift-invariant. By the definition \eqref{rate-def}, there is a sequence $(\mu_k)_{k=1}^\infty$ in $\mathcal M_\phi(\mathbb N^{\mathbb N},\sigma)$ that converges to $\mu$ in the weak* topology with
$\lim_k F_\phi(\mu_k)=0$. By \cite[Lemma~2.3]{Tak20}, we have $\inf_k\int\phi d\mu_k>-\infty$.
By this and $\sup\phi<\infty$, a simple upper semicontinuity argument as in \cite[Remark~2.5]{Tak20} shows $\int\phi d\mu>-\infty$.
If $\liminf_kh(\mu_k)=0$, then for any subsequence $(\mu_{k_j})_{j=1}^\infty$
with $\lim_jh(\mu_{k_j})=0$ we have
\[0=\lim_{j\to\infty}F_\phi(\mu_{k_j})\leq
\int\phi  d\mu\leq h(\mu)+\int\phi d\mu=F_\phi(\mu).\]
Since $F_\phi(\mu)\leq P(\phi)=0$, 
 $\mu$ is an equilibrium state for $\phi$.
If $\liminf_kh(\mu_k)>0$, then we have $\liminf_k(-\int\phi d\mu_k)>0$ and \[0=\lim_{k\to\infty}F_\phi(\mu_k)=\lim_{k\to\infty}\left(
-\int\phi d\mu_k\right)\left(\frac{h(\mu_k)}{-\int\phi d \mu_k}-1\right).\]
It follows that
\[\lim_{k\to\infty}\left(\frac{h(\mu_k)}{-\int\phi d\mu_k}-1\right)=0.\]
We have $-\int\phi d\mu\geq h(\mu)$.
If $-\int\phi d\mu=0$, then clearly $\mu$ is an equilibrium state for $\phi$.
If $-\int\phi d\mu>0$, then by \cite[Theorem~2.4]{Tak20} we have
\[\frac{h(\mu)}{-\int\phi d\mu}-1\geq0,\]
namely $F_\phi(\mu)\geq0$. Since $F_\phi(\mu)\leq0$, $\mu$ is an equilibrium state for $\phi$. The proof of Proposition~\ref{unique-prop} is complete. \end{proof}

\subsection{Symbolic coding of the Gauss-R\'enyi map}\label{coding}
The next proposition allows us to introduce a symbolic representation of the Gauss-R\'enyi map.

\begin{prop}\label{CF-unique} The following statements hold.
\begin{itemize}
\item[(a)] For every $(a_n)_{n\in\mathbb N}\in\mathbb N^{\mathbb N}$ we have
$\bigcap_{n=1}^\infty\varDelta(a_1\cdots a_n)=\{(\omega,x)\}\subset\Lambda$, where
 $\omega_n\equiv a_n\mod 2$, $C_n=(a_n+\omega_n)/2+\omega_{n+1}$ and
\[x=\omega_1+\confrac{(-1)^{\omega_1}}{C_1}  +\confrac{(-1)^{\omega_{2}}}{C_{2}}+\confrac{(-1)^{\omega_{3}}}{C_{3}}+\cdots.\]
\item[(b)] For every $(\omega,x)\in\Lambda$ we have $\{(\omega,x)\}=\bigcap_{n=1}^\infty\varDelta(a_1\cdots a_n)$, where
$a_n=2C_n(\omega,x)+\omega_n-2\omega_{n+1}$.

\end{itemize}
\end{prop}

\begin{proof}As for (a), let $(a_n)_{n\in\mathbb N}\in\mathbb N^{\mathbb N}$. Define $(\omega_n)_{n\in\mathbb N}\in\{0,1\}^{\mathbb N}$ by
$\omega_n\equiv a_n\mod 2$, and 
$C_n=(a_n+\omega_n)/2+\omega_{n+1}$
for $n\in\mathbb N$.
Note that  $(-1)^{\omega_{n+1}}+C_{n}\geq1$ for every $n\in\mathbb N$.
 By Lemma~\ref{IFS-new}, the displayed continued fraction 
 converges to a number $x\in[0,1]$, and thus $(\omega,x)\in
 \bigcap_{n=1}^\infty\varDelta(a_1\cdots a_n)$.
 The algorithm described in \S\ref{random-s} shows $\{(\omega,x)\}=
 \bigcap_{n=1}^\infty\varDelta(a_1\cdots a_n)$.
 Since
 $R^n(\omega,x)=(\theta^n\omega,T_{\omega}^nx)$ we have
\[T_{\omega}^nx=\omega_{n+1}+\confrac{(-1)^{\omega_{n+1}}}{C_{n+1}}  +\confrac{(-1)^{\omega_{n+2}}}{C_{n+2}}+\confrac{(-1)^{\omega_{n+3}}}{C_{n+3}}+\cdots.\]
Hence $(\omega,x)\in\Lambda$ holds.

To prove (b), let $(\omega,x)\in\Lambda$.  
Define $a_n=2C_n(\omega,x)-\omega_n-2\omega_{n+1}$ for $n\in\mathbb N$. We have $(-1)^{\omega_{n+1}}+C_{n}(\omega,x)\geq1$ for every $n\in\mathbb N$.
Proposition~\ref{random-lem}(a) gives \[x=\omega_1+\confrac{(-1)^{\omega_1}}{C_1(\omega,x)}  +\confrac{(-1)^{\omega_{2}}}{C_{2}(\omega,x)}+\confrac{(-1)^{\omega_3}}{C_3(\omega,x)}+\cdots,\]
which implies $(\omega,x)\in\bigcap_{n=1}^\infty\varDelta(a_1\cdots a_n)$. Proposition~\ref{CF-unique}(a) yields $\{(\omega,x)\}=\bigcap_{n=1}^\infty\varDelta(a_1\cdots a_n)$.
\end{proof}

Define a {\it coding map}
$\pi\colon\mathbb N^{\mathbb N}\to\Lambda$ by
\begin{equation}\label{code-def}\pi((z_n)_{n=1}^\infty)\in\bigcap_{n=1}^\infty\varDelta(z_1\cdots z_n).\end{equation}
By Proposition~\ref{CF-unique}, $\pi$ is well-defined and surjective.
 Obviously $\pi$ is continuous, injective and satisfies $R\circ \pi=\pi\circ\sigma$. It is not hard to show that $\pi$ maps Borel sets to Borel sets.
 We set
 \begin{equation}\label{P-def}\eta_p=(m_p\otimes\lambda_p)\circ\pi,\end{equation}
 and call $\eta_p$ the {\it annealed Gauss-R\'enyi measure}.
 From (b) and (c) in Proposition~\ref{random-lem}, we have $\Lambda_\omega=(0,1)\setminus\mathbb Q$ for every $\omega\in\Omega_0$. This implies  
$\Omega_0\times ((0,1)\setminus\mathbb Q)\subset\Lambda$, and so $(m_p\otimes\lambda_p)(\Lambda)=1$. Hence $\eta_p$ is a probability. 
The measure $m_p\otimes\lambda_p$ is $R$-invariant \cite[Theorem 3.2]{KKV17} and by \cite[Theorem~3.3]{KKV17} it is mixing. Hence
  $\eta_p$
is $\sigma$-invariant and mixing. 

By Lemma~\ref{represent-prop}, the induced system $\widehat\sigma\colon{\widehat{\mathbb N}}^{\mathbb N}\to{\widehat{\mathbb N}}^{\mathbb N}$ is topologically conjugate to $\sigma\colon\mathbb M^{\mathbb N}\to \mathbb M^{\mathbb N}$ via $\iota$.
Since $R\colon\Lambda\to\Lambda$ is topologically conjugate to $\sigma\colon\mathbb N^{\mathbb N}\to\mathbb N^{\mathbb N}$
 via $\pi$, the two induced systems $\widehat R\colon\widehat\Lambda\to\widehat\Lambda$ and  
  $\widehat\sigma\colon{\widehat{\mathbb N}}^{\mathbb N}\to{\widehat{\mathbb N}}^{\mathbb N}$
 are topologically conjugate 
 via $\pi$. 
The three dynamical systems are summarized in the following diagram.
\begin{equation}\label{diagram2}  \begin{CD}
     \mathbb M^{\mathbb N} @>{\sigma}>> \mathbb M^{\mathbb N} \\
  @V{\iota}VV    @VV{\iota}V \\
     {\widehat{\mathbb N}}^{\mathbb N}  @>\widehat\sigma>>  {\widehat{\mathbb N}}^{\mathbb N} \\
     @V{\pi}VV    @VV{\pi}V \\
     \widehat\Lambda  @>\widehat R>>  \widehat\Lambda \\
  \end{CD}\end{equation}



\subsection{Refined distortion estimates}\label{random-dist}

The distortion estimate in Lemma~\ref{dist-basic} does not suffice when $a_1\cdots a_n$ contains a long block of $1$ that contains $a_n$. The next lemma provides refined estimates in this case.
\begin{lemma}\label{Holder-d}
There exists a constant $K>0$ such that if $n\in\mathbb N$, 
$a_i=1$ for $i=1,\ldots,n$
and $a_{n+1}\neq1$ then 
for any pair $(\omega,x),(\varrho,y)$ of points in 
$\varDelta(a_1\cdots a_{n+1})$, 
\[S_n\varphi(\omega,x)-S_n\varphi(\varrho,y)\leq \begin{cases}K|T_\omega^{n}x-T_\varrho^{n}y|&\text{ if } a_{n+1}\in\mathbb N_1,\\
 K|T_\omega^nx-T_\varrho^n y|^{\frac{1}{2}}&\text{ if }a_{n+1}\in\mathbb N_0.\end{cases}\]

\end{lemma}

\begin{proof}
 Let $n\in\mathbb N$ and suppose
$a_i=1$ for $i=1,\ldots,n$
and $a_{n+1}\neq1$. 
For $i=0,\ldots,n$ put
\[q_i= \begin{cases}\vspace{1.3mm}\displaystyle{\frac{1}{i+2}}&\text{ if } a_{n+1}\in\mathbb N_1,\\
 \displaystyle{\frac{2}{2i+a_{n+1}}}&\text{ if }a_{n+1}\in\mathbb N_0,\end{cases}\]
 and
$J_i=\left[q_{i+1},q_i\right)$. Let
$(\omega,x),(\varrho,y)\in\varDelta(a_1\cdots a_{n+1})$. 
We have $T_1(q_{i+1})=q_i$ for $i=0,\ldots,n-1$ and 
$x,y\in J_{n-1}$. 
If $a_{n+1}\in\mathbb N_1$ then
by Lemma~\ref{scope'} below applied to $f=T_1|[0,1/2)$, there exists a uniform constant $K_1>0$ such that
\begin{equation}\label{ho-eq1}S_n\varphi(\omega,x)-S_n\varphi(\varrho,y)\leq K_1|T_\omega^{n}x-T_\varrho^{n}y|.\end{equation}
If $a_{n+1}\in\mathbb N_0$ then we have
\begin{equation}\label{Holder-d-eq1}
|J_0|=\frac{4}{a_{n+1}^2+2a_{n+1}}\ \text{ and }\ \sum_{i=0}^{n-1}|J_i|\leq \frac{2}{a_{n+1}}.
\end{equation}
By Lemma~\ref{scope'} below applied to the restriction $f=T_1|_{[0,2/a_{n+1})}$,
there exists a uniform constant $K_2>0$ such that
\[S_n\varphi(\omega,x)-S_n\varphi(\varrho,y)\leq K_2\frac{|T_\omega^nx-T_\varrho^ny|}{|J_0|}\sum_{i=0}^{n-1}
|J_i|.\]
Since
$R^n(\omega,x),R^n(\varrho,y)\in\varDelta(a_{n+1})$, the points
$T_\omega^nx$, $T_\varrho^ny$ belong to the closure of $J_0$, and thus
$|T_\omega^nx-T_\varrho^ny|/|J_0|\leq1$.
By this and \eqref{Holder-d-eq1},
\begin{equation}\label{ho-eq2}\begin{split}S_n\varphi(\omega,x)-S_n\varphi(\varrho,y)&\leq K_2\frac{|T_\omega^nx-T_\varrho^ny|}{|J_0|}\sum_{i=0}^{n-1}
|J_i|\\
&\leq K_2
\frac{|T_\omega^nx-T_\varrho^ny|^{\frac{1}{2}}}{|J_0|^{\frac{1}{2}} }\sum_{i=0}^{n-1}
|J_i|\\
&\leq K_2\frac{\sqrt{a_{n+1}^2+2a_{n+1}}}{a_{n+1}}|T_\omega^nx-T_\varrho^ny|^{\frac{1}{2}}\\
&\leq\sqrt{2}K_2|T_\omega^nx-T_\varrho^ny|^{\frac{1}{2}}.\end{split}\end{equation}
By \eqref{ho-eq1} and \eqref{ho-eq2}, taking $K=\max\{K_1,\sqrt{2}K_2\}$ yields the desired inequalities.
\end{proof}

The next general lemma on distortions for iterations of an interval map with a neutral fixed point was shown in the proof of \cite[Lemma~5.3]{JT}.

\begin{lemma}[{cf. \cite[Lemma~5.3]{JT}}]
\label{scope'}
Let $r>0$ and let $f\colon[0,r)\to \mathbb R$ be a $C^{2}$ map satisfying
$f0=0$, $f'0=1$ and $f'x>1$ for all $x\in(0,r)$. There exists a constant $K>0$ such that
for every $n\in\mathbb N$ and any pair $x,y$ of points in $J_{n-1}$,
\[\log\frac{|(f^n)'y|}{|(f^n)'x|}\leq 
K|f^nx-f^ny|\sum_{i=0}^{n-1}\frac{|J_{i}|}{|J_0|},\]
where 
$q_0=r$, $fq_{i+1}=q_{i}$ and $J_i=[q_{i+1},q_i)$ for $i=0,\ldots, n-1$.
\end{lemma}


We now proceed to distortion estimates of an induced potential.
 Notice that 
\[\widehat\Lambda=(\Lambda\setminus\varDelta(1))\setminus \bigcup_{n=1}^\infty R^{-n}((1^\infty,0)).\]
 Define an {\it induced annealed geometric potential} $\widehat\varphi\colon\widehat\Lambda\to\mathbb R$ by 
\[\widehat\varphi(\omega,x)=S_{t(\omega,x)}\varphi(\omega,x).\]
For a pair
$(\omega,x),(\varrho,y)$ of distinct points in $\widehat\Lambda$ contained in the same $1$-cylinder, we introduce their {\it separation time}
\[s((\omega,x),(\varrho,y))=\min\{n\geq1\colon a_1(\widehat R^n(\omega,x))\neq a_1(\widehat R^n(\varrho,y))\}.\]
Note that $s((\omega,x),(\varrho,y))\geq2$ implies $t(\omega,x)=t(\varrho,y)$.
We evaluate the quantity
\[\widehat\varphi(\omega,x)-\widehat\varphi(\varrho,y)=\log\frac{|(T_\omega^{t(\omega,x)})'y|}{|(T_\omega^{t(\omega,x)})'x|}.\]
\begin{lemma}\label{c-dist}
There exist constants $K>0$ and $\tau\in(0,1)$ such that 
 for any pair $(\omega,x),(\varrho,y)$ of points in $\widehat\Lambda$ with $s((\omega,x),(\varrho,y))\geq2$, 
\[\widehat\varphi(\omega,x)-\widehat\varphi(\varrho,y)\leq K\tau^{s((x,\omega),(\varrho,y))}.\]
\end{lemma}

\begin{proof}
For $(\omega,x),(\varrho,y)\in\widehat\Lambda$ as in the statement, put 
\[k=\min\{i\geq1\colon R^i(\omega,x)\in\varDelta(1)\}
\text{ and }  n=t(\omega,x),\] and decompose $R^n=R^{n-k}\circ R^k$. We estimate contributions from the first $k$ iteration  and the remaining $n-k$ iteration separately.
Lemma~\ref{dist-basic} gives
 \begin{equation}\label{c-dist-eq1}S_k\varphi(\omega,x)-
S_k\varphi(\varrho,y)\leq 2|T^k_\omega x-T^k_\varrho y|\ \text{ if  $k=1$}.
\end{equation}
By Lemma~\ref{dist-basic} and Lemma~\ref{3-exp},
\begin{equation}\label{c-dist-eq2}\begin{split}S_k\varphi(\omega,x)-
S_k\varphi(\varrho,y)&\leq 2\sum_{i=1}^k|T^i_\omega x-T^i_\varrho y|
\\
&\leq2\left(1+\sum_{i=1}^{k-1}\left(\frac{4}{9}\right)^{\lfloor(k-i)/2\rfloor}\right)|T^k_\omega x-T^k_\varrho y|\ \text{ if $k>1$}.\end{split}
\end{equation}
Put
$\tau=(4/9)^{\frac{1}{4}}\in(0,1)$ and
$K_0=2\left(1+\sum_{i=0}^{\infty}(4/9)^{\lfloor i/2\rfloor}\right)$.
 By the mean value theorem, 
 there exists 
 $(\theta^n\omega,z)\in \varDelta(a_{n+1}(\omega,x))$
such that
 \[ \begin{split}| T_{\omega}^kx-T_{\omega}^ky|&\leq
 |T_{\omega}^nx-T_{\omega}^ny|\\
 &=\frac{
|T_{\omega}^{\sum_{i=0}^{s((\omega,x),(\varrho,y))-1} t(\widehat R^i(\omega,x))}x-T_{\omega}^{\sum_{i=0}^{s((\omega,x),(\varrho,y))-1} t(\widehat R^i(\omega,x))}y|}{|(T_{\theta^n\omega}^{\sum_{i=0}^{s((\omega,x),(\varrho,y))-1} t(\widehat R^i(\omega,x))-n})'z|}.\end{split}\]
By Lemma~\ref{3-exp},
there exists a uniform constant $K_1>0$ such that
\begin{equation}\label{expansion-eq}|T_{\omega}^nx-T_{\omega}^ny|\leq \frac{1}{|(T_{\theta\omega}^{\sum_{i=0}^{s((\omega,x),(\varrho,y))-1} t(\widehat R^i(\omega,x))-n})'z|}\leq K_1\tau^{ 2s((\omega,x),(\varrho,y)) }.
\end{equation}
By Lemma~\ref{Holder-d}, there exists a uniform constant $K_2>0$ such that
\begin{equation}\label{expan-eq2}|S_{n-k}\varphi(R^k(\omega,x))-S_{ n-k }\varphi( R^k(\varrho,y))|\leq K_2|T_\omega^nx-T_\varrho^ny|^{\frac{1}{2}}.\end{equation}
Combining \eqref{c-dist-eq1}, \eqref{c-dist-eq2}, \eqref{expansion-eq} and \eqref{expan-eq2} we obtain
\[\begin{split}\widehat\varphi(\omega,x)-\widehat\varphi(\varrho,y)&=S_n\varphi(\omega,x)-S_n\varphi(\varrho,y)\\
&\leq |S_k\varphi(\omega,x)-S_k\varphi(\varrho,y)|+|S_{n-k}\varphi(R^k(\omega,x))-S_{n-k}\varphi(R^k(\varrho,y))|\\
&\leq K_0K_1\tau^{2s((\omega,x),(\varrho,y)) }+K_2|T_\omega^nx-T_\varrho^ny|^{\frac{1}{2}}\\
&\leq (K_0K_1+K_2\sqrt{K_1})\tau^{s((\omega,x),(\varrho,y))}.\end{split}\]
Setting $K=K_0K_1+K_2\sqrt{K_1}$ yields the desired inequality.
\end{proof}

For each $n\in\mathbb N$ define 
\[V_n(\widehat\varphi)=\sup\{\widehat\varphi(\omega,x)-\widehat\varphi(\varrho,y)\colon (\omega,x),(\varrho,y)\in\widehat\Lambda,\ s((\omega,x),(\varrho,y))\geq n\}.\]

 \begin{cor}\label{sum-cor}
 There exist constants $K>0$ and $\gamma\in(0,1)$ such that for every $n\geq1$ we have $V_n(\widehat\varphi)\leq K\gamma^n$.\end{cor}
 \begin{proof}
Follows from Lemma~\ref{Holder-d} and Lemma~\ref{c-dist}.\end{proof}

\subsection{Variational characterization of the annealed Gauss-R\'enyi measure}\label{variational-sec}
Define a potential $\psi\colon\mathbb N^{\mathbb N}\to\mathbb R$ by \begin{equation}\label{agp}\psi=\varphi\circ\pi\end{equation}
and an induced potential $\widehat\psi:\mathbb N^{\mathbb N}\setminus[1]\to\mathbb R$ by
\begin{equation}\label{agp-psi}\widehat\psi=\widehat\varphi\circ\pi|_{\mathbb N^{\mathbb N}\setminus[1]}.\end{equation}
\begin{lemma}\label{phi-p}
The potential $\psi$ is unbounded and $\sup\psi<0$. It is acceptable. 
\end{lemma}
\begin{proof}
The first assertion follows from the fact that $\varphi$ is unbounded and  $\sup\varphi<0$. The second one follows from R\'enyi's condition \eqref{ren} and Lemma~\ref{u-decay}.
\end{proof}

The annealed Gauss-R\'enyi measure $\eta_p$ has the so-called `weak Gibbs property'.
\begin{lemma}\label{gibb}There exists  $K\geq1$ such that
 for all $n\geq1$, all $a_1\cdots a_n\in\mathbb N^n$ and all $x\in [a_1\cdots a_n]$,
\[
K^{-1}\exp(-D_{n}(\varphi))\leq\frac{\eta_p([a_1\cdots a_n]}{\exp S_{n}\psi(x)}\leq K\exp(D_{n}(\varphi)).
\]\end{lemma}
\begin{proof}Follows from the fact that $h_p$ is bounded from above and away from $0$.\end{proof}

\begin{lemma}\label{z-finite}
We have $P(\psi)=0$.
\end{lemma}
\begin{proof}
By Lemma~\ref{gibb}, for all $n\geq1$ and all $a_1\cdots a_n\in \mathbb N^n$ we have
\[K^{-1}\exp(-D_n(\varphi))\eta_p([a_1\cdots a_n])\leq\sup_{[a_1\cdots a_n]} \exp S_n\psi\leq K\exp(D_n(\varphi))\eta_p([a_1\cdots a_n]).\]
Since $\eta_p$ is a probability and $n$-cylinders are pairwise disjoint,
summing the double inequalities over all $a_1\cdots a_n\in \mathbb N^n$, taking logarithms, dividing by $n$ and using
 Lemma~\ref{mild-lem} we obtain
 $P(\psi)=0$.
\end{proof}




By Lemma~\ref{phi-p} and Lemma~\ref{z-finite}, $\psi$ is acceptable and satisfies $P(\psi)<\infty$. By Proposition~\ref{thermo1}, the variational principle holds for $\psi$.
Due to the existence of the neutral fixed point of the R\'enyi map $T_1$, $\psi$ is not locally H\"older continuous. Nevertheless the following holds.
\begin{prop}\label{unique-equi}
The annealed Gauss-R\'enyi measure $\eta_p$ is the unique equilibrium state for the potential $\psi$.
\end{prop}

\begin{proof}
A proof of Proposition~\ref{unique-equi} breaks into two steps.
We first show that $\eta_p$ is an equilibrium state for the potential $\psi$. 
We then  
establish the uniqueness of equilibrium state
for the potential $\psi$.
To overcome the lack of regularity of $\psi$
in the second step,  
we take an inducing procedure that is now familiar in the construction of equilibrium states (see e.g., \cite[Section~8]{MauUrb03}, \cite{PesSen08}).

\medskip

\noindent{\bf Step 1: identifying $\eta_p$ as an equilibrium state.}
 Since $\log|T_0'|$ and  $\log|T_1'|$ are Lebesgue integrable, and since the Radon-Nikod\'ym derivative $h_p$ is bounded from above, $\psi$ is $\eta_p$-integrable. Since $P(\psi)$ is finite by Lemma~\ref{z-finite}, the measure-theoretic entropy $h(\eta_p)$ is finite (see \S\ref{general}).
The family of $1$-cylinders generates the Borel sigma algebra on $\mathbb N^{\mathbb N}$. Since $h_p$ is bounded from above and away from $0$, 
using the Lebesgue measure on $[0,1]$ and \eqref{a1-def} one can show that
$-\sum_{k\in \mathbb N } \eta_p([k])\log\eta_p([k])$ is finite.
Since $\eta_p$ is mixing, it is ergodic. The Shannon-McMillan-Breimann theorem yields
 \[\lim_{n\to\infty}\frac{1}{n}\log\eta_p([x_1\cdots x_n])=-h(\eta_p)\ \text{ $\eta_p$-a.e.}\]  
Meanwhile, from Lemma~\ref{gibb} and Lemma~\ref{mild-lem}
 it follows that
 \[\lim_{n\to\infty}\frac{1}{n}\log\eta_p([x_1\cdots x_n])=\int\psi d\eta_p\ \text{ $\eta_p$-a.e.}\]
We have verified that
 $h(\eta_p)+\int\psi d\eta_p=0$. Since $P(\psi)=0$ by Lemma~\ref{z-finite}, $\eta_p$ is an equilibrium state for $\psi$.
 \medskip

 \noindent{\bf Step~2: establishing the uniqueness of equilibrium state.}
Recall that
 $\widehat\sigma\colon{\widehat{\mathbb N}}^{\mathbb N}\to{\widehat{\mathbb N}}^{\mathbb N}$ is the induced system associated with
the inducing scheme $(\mathbb N^{\mathbb N}\setminus[1],t_{\mathbb N^{\mathbb N}\setminus[1]})$
of the left shift
 $\sigma\colon\mathbb N^{\mathbb N}\to\mathbb N^{\mathbb N}$ (see \S\ref{ind-sec}).
 For the induced potential $\widehat\psi$ in \eqref{agp-psi}, 
  define  $\Psi\colon \mathbb M^{\mathbb N}\to\mathbb R$ by \[\Psi=\widehat\psi\circ\iota.\] 
\begin{lemma}\label{holder}
The potential $\Psi$ is locally H\"older continuous. 
\end{lemma}
\begin{proof}Follows from Corollary~\ref{sum-cor}.\end{proof}



Next we compute the pressure $P(\Psi)$.
\begin{lemma}\label{zero-p}
We have 
$P(\Psi)=0$. 
\end{lemma}
\begin{proof}
Put $K_0=\sum_{n=1}^\infty {\rm var}_n(\Psi)$.
By Lemma~\ref{holder}, $K_0$ is finite. For all $n\geq1$ and all $\alpha_1\cdots \alpha_n\in \mathbb M^{n}$ we have 
\[\sup_{\eta,\zeta\in[\alpha_1\cdots \alpha_n]}\left( S_n\Psi(\eta)-S_n\Psi(\zeta)\right)\leq \sum_{k=1}^n {\rm var}_k(\Psi)\leq K_0.\]
Since $h_p$ is bounded from above and away from $0$, there is a constant $K_1\geq1$ such that 
for all $n\geq1$ and all $\alpha_1\cdots \alpha_n\in \mathbb M^{n}$, we have 
\[K_1^{-1}\eta_p([\alpha_1\cdots \alpha_n])\leq
\sup_{[\alpha_1\cdots \alpha_n] } 
\exp S_n\Psi \leq K_1\eta_p([\alpha_1\cdots \alpha_n]).\] 
Summing these double inequalities over all $\alpha_1\cdots \alpha_n\in\mathbb M^n$,  
\[K_1^{-1}\sum_{\alpha_1\cdots \alpha_n\in\mathbb M^n}\eta_p([\alpha_1\cdots \alpha_n])\leq
\sum_{\alpha_1\cdots \alpha_n\in\mathbb M^n}\sup_{[\alpha_1\cdots \alpha_n]} 
\exp S_n\Psi \leq K_1.\]
By the definition of $\widehat\Lambda$ and the fact that $m_p\otimes\lambda_p$ has no atom,  
\[\sum_{\alpha_1\cdots \alpha_n\in \mathbb M^n}\eta_p([\alpha_1\cdots \alpha_n])=\eta_p(\Sigma)=(m_p\otimes\lambda_p)(\widehat\Lambda)=(m_p\otimes\lambda_p)(\Lambda\setminus\varDelta(1))>0.\]  
Hence, taking logarithms of the above double inequalities, dividing the result by $n$ and letting $n\to\infty$ yields
 $P(\Psi)=0$. 
\end{proof}

Since $\Psi$ is acceptable by Lemma~\ref{holder} and $P(\Psi)$ is finite by Lemma~\ref{zero-p}, the variational prinicple holds by Proposition~\ref{thermo1}.
By Proposition~\ref{thermo2} and $P(\Psi)=0$ from Lemma~\ref{zero-p}, 
there exists a unique shift-invariant Gibbs state $\widehat\mu\in\mathcal M(\mathbb M^{\mathbb N},\sigma)$, namely, 
 there exists a constant $K\geq1$ such that 
for every $n\geq1$, every $\alpha_1\cdots \alpha_n\in \mathbb M^n$ and every $z\in[\alpha_1\cdots \alpha_n]$,
\begin{equation}\label{gibbs}
K^{-1}\leq\frac{\widehat\mu([\alpha_1\cdots \alpha_n])}{\exp S_n\Psi(z)}\leq K
.\end{equation}


\begin{lemma}\label{finite} Both
$\int t_{\mathbb N^\mathbb N\setminus[1]}\circ
\iota d\widehat\mu$ and $\int\Psi d\widehat\mu$ are finite.\end{lemma}
\begin{proof}
The function $ t_{\mathbb N^\mathbb N\setminus[1]}\circ\iota$ is constant on $[\alpha]$ for each $\alpha\in \mathbb M$. Let $t_\alpha$ denote this constant. By the second inequality in
 \eqref{gibbs}, for all $(\omega,x)\in \pi\circ\iota([\alpha])$ we have
\[\widehat\mu([\alpha])\leq K(1-p)p^{t_\alpha-1}|(T_\omega^{t_\alpha})'x|^{-1}\leq K(1-p)p^{t_\alpha-1}|T_\omega'x|^{-1}.\]
For every $k\in\mathbb N\setminus\{1\}$, there is $\alpha\in\mathbb M$ such that $\pi([\alpha])\subset\varDelta(k)$
and $t_\alpha=n$. Hence
\begin{equation}\label{mu-hat}\begin{split}\sum_{\substack{\alpha\in\mathbb M \\ t_\alpha=n}}\widehat\mu([\alpha])&\leq K(1-p)p^{n-1}
\left(\sum_{k=1}^\infty\sup_{ \varDelta(2k)}|T_0'|^{-1}+\sum_{k=2}^\infty\sup_{\varDelta(2k-1)}|T_1'|^{-1}\right)\\
&\leq  2e^2K(1-p)p^{n-1}\left(\sum_{k=1}^\infty
|J(2k)|+\sum_{k=2}^\infty|J(2k-1)|\right)\\
&= 3e^2K(1-p)p^{n-1}.\end{split}\end{equation}
To deduce the second inequality we have used \eqref{ren}. Therefore
\[\int t_{\mathbb N^\mathbb N\setminus[1]}\circ
\iota d\widehat\mu=\sum_{n=1}^\infty n\sum_{\substack{\alpha\in\mathbb M \\ t_\alpha=n}}\widehat\mu([\alpha])<\infty,
\]
as required.

There exist constants $K>0$ and $c>1$ such that if $n\in \mathbb N$ and $x\in J(1)$ are such that $x,\ldots,T_1^{n-1}x\in J(1)$ then
$|(T_{1}^{n})'x|\leq Kc^{n}$. Moreover, $c$ can be taken arbitrarily close to $1$ at the expense of enlarging $K$.
 Now, let $n\in\mathbb N$, $\alpha\in\mathbb M$ satisfy $t_\alpha=n$. For $\zeta=(\omega,x)\in[\alpha]$ we have
\[\Psi(\zeta)=\log p(\omega_1)-\log|(T_{\omega_1})'x|+(n-1)\log p-\log|(T_{1}^{n-1})'T_{\omega_1}x|,\] 
where $T_{\omega_1}x,\ldots,T_{\omega_1}^{n-1}x\in J(1)$ provided $n\geq2$. 
It follows that 
there exists a constant $K>0$ independent of $n$, $\alpha$, $\zeta$ such that  
\begin{equation}\label{mu-hat2}|\Psi(\zeta)|\leq Kn.\end{equation} 
From \eqref{mu-hat} and \eqref{mu-hat2} we obtain
\[\left|\int\Psi d\widehat\mu\right|\leq
\int|\Psi| d\widehat \mu\leq \sum_{n=1}^\infty\sum_{\substack{\alpha\in\mathbb M \\ t_\alpha=n}}\widehat\mu([\alpha])\sup_{[
\alpha]}|\Psi|
\leq\sum_{n=1}^\infty Kn\sum_{\substack{\alpha\in\mathbb M \\ t_\alpha=n}}\widehat\mu([\alpha])<\infty,\]
as required.
\end{proof}

Since $\int\Psi d\widehat\mu$ is finite
by Lemma~\ref{finite},
 $\widehat\mu$ is the unique equilibrium state for the potential 
$\Psi$ by Proposition~\ref{thermo2}. In particular we have
\begin{equation}\label{comb1}P(\Psi)=h(\widehat\mu)+\int\Psi d
\widehat \mu.\end{equation}
By the finiteness of $\int t_{\mathbb N^\mathbb N\setminus[1]}\circ\iota d\widehat\mu$ in Lemma~\ref{finite}, 
the measure
\[\mu=\frac{1}{\int  t_{\mathbb N^\mathbb N\setminus[1]}\circ\iota  d\widehat\mu}\sum_{n=1}^\infty\sum_{i=0}^{n-1}
\widehat\mu|_{\{ t_{\mathbb N^\mathbb N\setminus[1]}\circ\iota=n\}}\circ\iota^{-1}\circ \sigma^{-i} \] belongs to $\mathcal M(\mathbb N^{\mathbb N},\sigma),$ and
by Abramov-Kac's formula \cite[Theorem~2.3]{PesSen08}
\begin{equation}\label{comb2}h(\widehat\mu)+\int\Psi d
\widehat \mu=\left(h(\mu)+\int\psi d\mu\right)\int  t_{\mathbb N^\mathbb N\setminus[1]}\circ\iota d\widehat\mu.\end{equation}
Combining \eqref{comb1}, \eqref{comb2} and 
 $P(\Psi)=0$ in Lemma~\ref{zero-p}
we obtain
$h(\mu)+\int\psi d\mu=0$.
Since $P(\psi)=0$ by Lemma~\ref{z-finite}, 
 $\mu$ is an equilibrium state for the potential $\psi$.


 We claim that $\mu$ is the unique equilibrium state for the potential $\psi$. 
 Indeed, let
 $\nu\in\mathcal M_\psi(\mathbb N^{\mathbb N},\sigma)$ be an equilibrium state for $\psi$ with $\nu(\widehat{\mathbb N}^\mathbb N)>0$.
  The normalized restriction of $\nu$ to $\widehat{\mathbb N}^\mathbb N$, denoted by $\widehat\nu$, belongs
 to $\mathcal M(\widehat{\mathbb N}^\mathbb N,\widehat\sigma_{\widehat{\mathbb N}^\mathbb N})$.
 From $P(\psi)=0$,
 Abramov-Kac's formula and
 $P(\Psi)=0$,
 $\widehat\nu$
 is an equilibrium state for the potential $\Psi$,
 namely $\widehat\mu=\widehat\nu$. 
It follows that $\mu=\nu$.
 Moreover, the only measure in $\mathcal M_\psi(\mathbb N^{\mathbb N},\sigma)$ which does not give positive weight to $\widehat{\mathbb N}^\mathbb N$ is the unit point mass
 at $\pi^{-1}(1^\infty,0)$, which is precisely the fixed point of $\sigma$ in the $1$-cylinder $[1]$. Since $h(\delta_{\pi^{-1}(1^\infty,0)})=0$ and  $|T_1'0|=1$, we have
 $h(\delta_{\pi^{-1}(1^\infty,0)})+\int\psi
 d\delta_{\pi^{-1}(1^\infty,0)}=\log p<0=P(\psi).$ Therefore
 the claim holds.
      The proof of Proposition~\ref{unique-equi} is complete.
\end{proof}



\subsection{Proof of Theorem~\ref{level-2-thm}}\label{section4-1}
We
define a sequence
$(\tilde\nu_n)_{n=1}^\infty$ of Borel probability measures on
$\mathcal M(\mathbb N^{\mathbb N})$ replacing $\phi$ in \eqref{nu_n} by $\psi$ in \eqref{agp}.
Define a parametrized family of twisted induced potentials $\Psi_\gamma\colon\mathbb M^{\mathbb N}\to\mathbb R$ $(\gamma\in\mathbb R)$ replacing $\phi$ in \eqref{ind-po} by $\psi$.
 Then $\Psi_\gamma$ is locally H\"older continuous for all $\gamma\in\mathbb R$ by Lemma~\ref{holder}, and $P(\Psi_0)=0$ by
Lemma~\ref{zero-p}. By Theorem~\ref{LDP-per}, 
 $(\tilde\nu_n)_{n=1}^\infty$  is exponentially tight and satisfies the LDP with the good rate function $I_\psi$.

The coding map $\pi\colon\mathbb N^{\mathbb N}\to\Lambda$ in \eqref{code-def} induces a continuous map
$\pi_*\colon\nu\in\mathcal M(\mathbb N^{\mathbb N})\mapsto\nu\circ\pi^{-1}\in\mathcal M(\Lambda)$.
 Since $\tilde\nu_n\circ\pi_*^{-1}=\tilde\mu_n$ for every $n\geq1$, by the Contraction Principle in Proposition~\ref{CP}, $(\tilde\mu_n)_{n=1}^\infty$ is exponentially tight and satisfies the LDP with the good rate function $I_p$ given by
 \[I_p(\mu)=\inf\{I_\psi(\nu)\colon \nu\in \mathcal M(\mathbb N^{\mathbb N}),\ \pi_*(\nu)=\mu\}.\]
 Since $I_\psi$ is convex, so is $I_p$.
 Since
  $\eta_p$ is an equilibrium state for $\psi$ by Proposition~\ref{unique-equi}, it is a minimizer of $I_\psi$. The equation 
  $\pi_*(\eta_p)=m_p\otimes\lambda_p$ shows that
  $m_p\otimes\lambda_p$ is a minimizer of $I_p$.

 By the last assertion of Proposition~\ref{CP}, 
to conclude the uniqueness of minimizer of $I_p$
it suffices to show the uniqueness of minimizer of $I_\psi$. Since $\psi$ is acceptable by Lemma~\ref{phi-p}, it is uniformly continuous. By virtue of Proposition~\ref{unique-prop}, 
it suffices to show $\beta_\infty(\psi)<1$.
Direct calculations show that
there exist constants $K_1>K_0>0$
such that 
\[\frac{4K_0 (1-p)}{k(k+2)}\leq\sup_{[k]}e^{\psi}\leq \frac{4K_1(1-p)}{k(k+2)}\]
for all $k\in\mathbb N_0$, and
\[\frac{4K_0p}{(k+1)(k+3)}\leq\sup_{[k]}e^{\psi}\leq \frac{4K_1p}{(k+1)(k+3)}\]
for all $k\in\mathbb N_1$.
Since $\sup_{[k]}e^{\beta\psi}=(\sup_{[k]}e^{\psi})^\beta$, 
these estimates imply $\beta_\infty(\psi)=1/2$.

The deduction of Theorem~\ref{level-2-thm}(b)
from Theorem~\ref{level-2-thm}(a) is much simpler than that of 
Theorem~\ref{ldpup-q}(b)
from Theorem~\ref{ldpup-q}(a) carried out in \S\ref{pf-sample}. 
The exponential tightness in 
Theorem~\ref{level-2-thm}(a) implies the tightness, which ensures the existence of a limit point by Prohorov's theorem. The LDP and the uniqueness of minimizer in Theorem~\ref{level-2-thm}(a) together rule out the existence of a limit point that is different from the unit point mass at the minimizer.
The proof of Theorem~\ref{level-2-thm} is complete.
\qed

\subsection{Annealed and quenched level-1 large deviations for the Gauss-R\'enyi map}\label{level-1-sec}
For $p\in(0,1)$ and a bounded continuous function $f\colon\Lambda\to\mathbb R$, define a function $I_{p,f}\colon\mathbb R\to[0,\infty]$ by
\[I_{p,f}(\alpha)=\inf\left\{I_p(\nu)\colon\nu\in\mathcal M(\Lambda),\ \int f d\nu=\alpha\right\}.\]
By Theorem~\ref{level-2-thm}(a), 
 $I_{p,f}$ is convex and vanishes
only at the mean $\alpha=\int f d(m_p\otimes\lambda_p)$.
Put
\[\underline{f}=\inf\left\{\int f d\nu\colon\nu\in\mathcal M(\Lambda)\right\}\ \text{ and }\ \overline{f}=\sup\left\{\int f d\nu\colon\nu\in\mathcal M(\Lambda)\right\}.\]
The next corollary of independent interest follows from 
the Contraction Principle applied to the level-2 LDP in Theorem~\ref{level-2-thm}(a).
\begin{cor}[annealed level-1 LDP]\label{lyacor}
Let
$f\colon\Lambda\to\mathbb R$ be a bounded continuous function
 such that $\underline{f}<\overline{f}$. For any $p\in(0,1)$ the following statements hold: 
 \begin{itemize}
\item[(a)]   if $\int f d(m_p\otimes\lambda_p)<\alpha\leq\overline f$ then \[\lim_{n\to\infty}\frac{1}{n}\log\sum_{\substack{(\omega,x)\in{\rm Fix}(R^n)\\ (1/n)\sum_{i=0}^{n-1}f( R^i(\omega,x))\geq\alpha}} Q_p^n(\omega)|(T^n_\omega)'x|^{-1}
      =-I_{p,f}(\alpha)<0;\]

\item[(b)]  if $\underline f\leq\alpha<\int f d(m_p\otimes\lambda_p)$ then
\[\lim_{n\to\infty}\frac{1}{n}\log\sum_{\substack{(\omega,x)\in{\rm Fix}(R^n)\\ (1/n)\sum_{k=0}^{n-1}f( R^k(\omega,x))\leq\alpha}} Q_p^n(\omega)|(T^n_\omega)'x|^{-1}
      =-I_{p,f}(\alpha)<0.\]
      \end{itemize}    \end{cor}

      We apply Corollary~\ref{lyacor} to  the problem of frequency of digits in the random continued fraction expansion \eqref{r-expansion}. 
      Recall the algorithm in \S\ref{random-s}, and let us use the square bracket to denote the $2$-cylinders in $\Omega$: for $i,j\in\{0,1\}$, \[[ij]=\{\omega\in\Omega\colon\omega_1=i,\omega_2=j\}.\] 
Let $n\in\mathbb N$ and $(\omega,x)\in\Lambda$. For each $k\in\mathbb N$, 
  $C_n(\omega,x)=k$ holds if and only if $C(R^{n-1}(\omega,x))=k$ and $\omega_{n+1}=0$, or else 
  $C(R^{n-1}(\omega,x))=k-1$ and $\omega_{n+1}=1$.
  For each $m\in\mathbb N$, 
  $C(\omega,x)=m$ holds if and only if $\lfloor1/x\rfloor=m$ and $\omega_1=0$, or else
  $\lfloor1/(1-x)\rfloor=m$ and $\omega_1=1$.

  If $k=1$ then define 
  \[A_k=[00]\times\left(\frac{1}{k+1},\frac{1}{k}\right].\]
  If $k\geq2$ then define
  \[\begin{split}A_k=&\left([00]\times\left(\frac{1}{k+1},\frac{1}{k}\right]\right)\cup
\left([10]\times\left[\frac{k-1}{k},\frac{k}{k+1}\right)\right)\\
&\cup
\left([01]\times\left(\frac{1}{k},\frac{1}{k-1}\right]\right)\cup
\left([11]\times\left[\frac{k-2}{k-1},\frac{k-1}{k}\right)\right).\end{split}\]
Notice that
$C_n(\omega,x)=k$ holds if and only if $R^{n-1}(\omega,x)\in A_k$. 
Let $\1_{k}\colon\Lambda\to\mathbb R$ denote the indicator function of $A_k\cap\Lambda$.
Let $p\in(0,1)$.
 By Birkhoff's ergodic theorem, 
for $m_p\otimes\lambda_p$-almost every $(\omega,x)\in\Lambda$ we have
\[\lim_{n\to\infty}\frac{\#\{1\leq i\leq n\colon C_i(\omega,x)=k\}}{n}=\int \1_kd(m_p\otimes\lambda_p).\]
Clearly, $\1_k$ is bounded continuous and satisfies 
$\underline{\1_k}=0$, $\overline{\1_k}=1$, 
$0<\int  \1_{k} d(m_p\otimes\lambda_p)<1$. By Corollary~\ref{lyacor} the following hold:
 \begin{itemize}
\item   if $\int  \1_{k} d(m_p\otimes\lambda_p)<\alpha\leq 1$ then \[\lim_{n\to\infty}\frac{1}{n}\log\sum_{\substack{(\omega,x)\in{\rm Fix}(R^n)\\ \frac{\#\{1\leq i\leq n\colon C_i(\omega,x)=k\}}{n}\geq\alpha}} Q_p^n(\omega)|(T^n_\omega)'x|^{-1}
      =-I_{p,\1_k}(\alpha)<0;\]

\item  if $0\leq\alpha<\int\1_k d(m_p\otimes\lambda_p)$ then
\[\lim_{n\to\infty}\frac{1}{n}\log\sum_{\substack{(\omega,x)\in{\rm Fix}(R^n)\\ \frac{\#\{1\leq i\leq n\colon C_i(\omega,x)=k\}}{n}\leq\alpha}} Q_p^n(\omega)|(T^n_\omega)'x|^{-1}
      =-I_{p,\1_k}(\alpha)<0.\]
      \end{itemize}    

Recall the notation in \S\ref{expansion-sec}. 
If $n\geq2$ then the indicator function of $A_k$
 is constant on each $n$-cylinder $\varDelta(a_1\cdots a_n)$. Moreover, each $n$-cylinder  contains exactly one point from ${\rm Fix}(R^n)$, and if $(\omega,x)\in\varDelta(a_1\cdots a_n)\cap{\rm Fix}(R^n)$ then by Lemma~\ref{mild-lem},
$Q_p^n(\omega)|(T^n_\omega)'x|^{-1}$ is comparable to $(m_p\otimes\lambda_p)(\varDelta(a_1\cdots a_n))$ up to the  subexponential factor $\exp(D_n(\varphi))$. Hence, the above annealed level-1 LDP for periodic points of $R$ extends to an annealed level-1 LDP for $m_p\otimes\lambda_p$-typical points:

\begin{itemize}
\item      if $\int  \1_{k} d(m_p\otimes\lambda_p)<\alpha\leq 1$ then
       \[\lim_{n\to\infty}\frac{1}{n}\log (m_p\otimes\lambda_p)\left\{(\omega,x)\in\Lambda\colon \frac{\#\{1\leq i\leq n\colon C_i(\omega,x)=k\}}{n}\geq\alpha\right\} 
      =-I_{p,\1_k}(\alpha);\]

 \item     if $0\leq\alpha<\int  \1_{k} d(m_p\otimes\lambda_p)$ then
       \[\lim_{n\to\infty}\frac{1}{n}\log (m_p\otimes\lambda_p)\left\{(\omega,x)\in\Lambda\colon \frac{\#\{1\leq i\leq n\colon C_i(\omega,x)=k\}}{n}\leq\alpha\right\} 
      =-I_{p,\1_k}(\alpha).\]
\end{itemize}

We now move on to a quenched counterpart.
The next corollary of independent interest is a consequence of 
 Theorem~\ref{ldpup-q}(a). Since it only gives an upper bound for closed sets, 
we only get inequalities for upper limits which should not be optimal.
\begin{cor}[quenched level-1 upper bounds]\label{lyacor2}
Let
$f\colon\Lambda\to\mathbb R$ be a bounded continuous function
 such that $\underline{f}<\overline{f}$. For any $p\in(0,1)$ the following statements hold: 
 \begin{itemize}
\item[(a)]   if $\int f d(m_p\otimes\lambda_p)<\alpha\leq\overline f$ then for $m_p$-almost every $\omega\in\Omega$, \[\limsup_{n\to\infty}\frac{1}{n}\log\sum_{\substack{x\in{\rm Fix}(T_\omega^n)\\ (1/n)\sum_{i=0}^{n-1}f( T_\omega^ix)\geq\alpha}} |(T^n_\omega)'x|^{-1}
      \leq-I_{p,f}(\alpha)<0;\]
      \item[(b)]   if $\underline{f}\leq\alpha<\int f d(m_p\otimes\lambda_p)$ then for $m_p$-almost every $\omega\in\Omega$, \[\limsup_{n\to\infty}\frac{1}{n}\log\sum_{\substack{x\in{\rm Fix}(T_\omega^n)\\ (1/n)\sum_{i=0}^{n-1}f( T_\omega^ix)\leq\alpha}} |(T^n_\omega)'x|^{-1}
      \leq-I_{p,f}(\alpha)<0.\]
      \end{itemize}    \end{cor}

Let $p\in(0,1)$ and $k\in\mathbb N$.
By Birkhoff's ergodic theorem and Fubini's theorem, 
for $m_p$-almost every $\omega\in\Omega$ and $\lambda_p$-almost every $x\in\Lambda_\omega$ we have
\[\lim_{n\to\infty}\frac{\#\{1\leq i\leq n\colon C_i(\omega,x)=k\}}{n}=\int \1_kd(m_p\otimes\lambda_p).\]
Corollary~\ref{lyacor2} yields the following:
 \begin{itemize}
\item   if $\int  \1_{k} d(m_p\otimes\lambda_p)<\alpha\leq 1$ then for $m_p$-almost every $\omega\in\Omega$, \[\limsup_{n\to\infty}\frac{1}{n}\log\sum_{\substack{x\in{\rm Fix}(T_\omega^n)\\ \frac{\#\{1\leq i\leq n\colon C_i(\omega,x)=k\}}{n}\geq\alpha}} |(T^n_\omega)'x|^{-1}
      \leq-I_{p,\1_k}(\alpha);\]

\item  if $0\leq\alpha<\int\1_k d(m_p\otimes\lambda_p)$ then for $m_p$-almost every $\omega\in\Omega$, 
\[\limsup_{n\to\infty}\frac{1}{n}\log\sum_{\substack{x\in{\rm Fix}(T_\omega^n)\\ \frac{\#\{1\leq i\leq n\colon C_i(\omega,x)=k\}}{n}\leq\alpha}} |(T^n_\omega)'x|^{-1}
      \leq-I_{p,\1_k}(\alpha).\]
      \end{itemize}

 Recall the notation in \S\ref{expansion-sec} again.
 Let $\omega\in\Omega$, $n\in\mathbb N$ and let $a_1\cdots a_n\in\mathbb N^{\mathbb N}$ satisfy
 $\omega_i\equiv a_i$ mod $2$ for $i=1,\ldots,n$. If $n\geq2$ then the restriction of the indicator function of $A_k$ to $\{\omega\}\times J(a_1\cdots a_n)$ is constant. Clearly,
 $J(a_1\cdots a_n)\cap {\rm Fix}(T_\omega^n)$ is a singleton.
 If $x\in J(a_1\cdots a_n)\cap{\rm Fix}(T_\omega^n)$, then by Lemma~\ref{mild-lem},
$|(T^n_\omega)'x|^{-1}$ is comparable to $\lambda_p(J(a_1\cdots a_n))$ up to the  subexponential factor $\exp(D_n(\varphi))$. Hence, the above quenched level-1 upper bounds extend to quenched level-1 upper bounds for $\lambda_p$-typical points:

\begin{itemize}
\item      if $\int  \1_{k} d(m_p\otimes\lambda_p)<\alpha\leq 1$ then for $m_p$-almost every $\omega\in\Omega$, 
       \[\limsup_{n\to\infty}\frac{1}{n}\log \lambda_p\left\{x\in(0,1)\setminus\mathbb Q\colon \frac{\#\{1\leq i\leq n\colon C_i(\omega,x)=k\}}{n}\geq\alpha\right\} 
      \leq-I_{p,\1_k}(\alpha);\]

 \item     if $0\leq\alpha<\int  \1_{k} d(m_p\otimes\lambda_p)$ then for $m_p$-almost every $\omega\in\Omega$, 
       \[\limsup_{n\to\infty}\frac{1}{n}\log \lambda_p\left\{x\in(0,1)\setminus\mathbb Q\colon \frac{\#\{1\leq i\leq n\colon C_i(\omega,x)=k\}}{n}\leq\alpha\right\} 
      \leq-I_{p,\1_k}(\alpha).\]
\end{itemize}

\appendix

\def\thesection{\Alph{section}}

\section{Periodic continued fractions}\label{determ-p}

The classical Lagrange theorem asserts that the regular continued fraction expansion of a quadratic irrational is eventually periodic. So, any quadratic irrational in $(0,1)$ is eventually periodic under the iteration of the Gauss map.
This appendix is a brief summary of known characterizations of periodic continued fractions in terms of iterations of the Gauss and R\'enyi maps.
For a quadratic irrational $x\in\mathbb R$, let $x^\dagger$ denote its Galois conjugate.
\begin{prop}[\cite{Gal}]\label{T0}
Let $x\in(0,1)$. The following are equivalent:
\begin{itemize}
    \item[(a)] $x$ is a quadratic irrational and $x^\dagger<-1$.
    \item[(b)] There exists $n\in\mathbb N$ such that $T_0^nx=x$.
\end{itemize}
\end{prop} 
Although much less known, 
statements analogous to Proposition~\ref{T0} hold for the R\'enyi map.
\begin{prop}\label{T1}
Let $x\in(0,1)$. 
The following are equivalent:

\begin{itemize}

\item[(a)] $x$ is a quadratic irrational and $x^\dagger<0$.

\item[(b)] There exists $n\in\mathbb N$ such that $T_1^nx=x$.
\end{itemize}
\end{prop}
For the reader's convenience we include a proof of Proposition~\ref{T1} below.
The idea is to translate 
analogous statements in \cite{Katok01} on the minus continued fraction to the backward continued fraction via simple algebraic manipulations.

Let $x\in\mathbb R$. We define a sequence $(x_n)_{n=0}^\infty$ of real numbers by 
\[x_0=x\ \text{ and }\ x_n=\frac{1}{\lfloor x_{n-1}\rfloor+1-x_{n-1}}\ \text{ for }n\geq1.\]
For $n\geq0$ put 
\[D_n(x)=\lfloor x_n\rfloor+1.\]
 For $n\geq1$, note that $D_n(x)\geq2$ 
since $x_n\geq1$.
For $n\geq1$ we set 
\[r_n(x)=D_0(x)-\confrac{1 }{D_{1}(x)} - \cdots -\confrac{1 }{D_{n}(x)}.\]
By \cite[Theorem~1.1]{Katok01} we obtain $x=\lim_n r_n(x)$, which is the {\it minus continued fraction expansion} of $x$: 
\[x=D_0(x)-\confrac{1 }{D_{1}(x)}-\confrac{1}{D_2(x)}-\cdots -\confrac{1 }{D_{n}(x)}-\cdots.\]
We say $x$ has a {\it purely periodic minus continued fraction expansion} of period $N+1$ if there exists $N\in\mathbb N$ such that 
\[x=D_0(x)-\confrac{1 }{D_{1}(x)}-\confrac{1}{D_2(x)}-\cdots -\confrac{1 }{D_{N}(x)}-\confrac{1}{x}.\]

\begin{prop}[{\cite[Theorem~1.4]{Katok01}}]\label{katok-lem}
    Let $x\in\mathbb R$ be a quadratic irrational. Then $x$ has a purely periodic minus continued fraction expansion if and only if $x > 1$ and $0 < x^\dagger < 1$.
    \end{prop}

\begin{proof}[Proof of Proposition~\ref{T1}]
    Let $x\in(0,1)$ be a quadratic irrational. 
    There is a quadratic equation $az^2+bz+c=0$ with integer coefficients whose solutions are $x,x^\dagger$.
This equation is equivalent to
$a(1-z)^2-(b+2a)(1-z)+(a+b+c)=0$.
We have $a+b+c\neq0$, for otherwise
$z=1$ would be a solution of the equation.
For $z\in\{x,x^\dagger\}$ we have
\[(a+b+c)\Bigl((1-z)^{-1}\Bigr)^2-(b+2a)(1-z)^{-1}+a=0.\]
Hence, 
$(1-x)^{-1}$ is a quadratic irrational whose Galois conjugate is $(1-x^\dagger)^{-1}$. 

Let $x\in(0,1)$ be a quadratic irrational and 
suppose $x^\dagger<0$.
Then $0<(1-x^\dagger)^{-1}<1$ holds.
Since $(1-x)^{-1}>1$,
by Proposition~\ref{katok-lem} there exists an integer $n\geq2$ such that the minus continued fraction expansion of  $(1-x)^{-1}$ is periodic of period of $n$:
\[
\frac{1}{1-x}=D_0(x)-\confrac{1 }{D_{1}(x)} -\cdots- \confrac{1 }{D_{n-1}(x)}  -\cdots\-\confrac{1}{D_0(x)}-\cdots-\confrac{1}{D_{n-1}(x)}-\cdots,
\]
where $D_i(x)\geq2$ for $i=0,\ldots,n-1$. Rearranging this equality gives 
\[x=1-\confrac{1}{D_0(x)}-\cdots-\confrac{1}{D_{n-1}(x)}-\confrac{1}{D_0(x)}-\cdots.\]
From this and the uniqueness of the backward continued fraction given by the R\'enyi map $T_1$, we obtain $T^{n}_1x=x$.

Conversely, suppose there exists $n\in\mathbb N$ such that $T_1^nx=x$. Then the backward continued fraction of $x$ given by $T_1$ is periodic of period $n$, and we have
\begin{equation*}
    x=
1-\confrac{1 }{B_{1}(x)} -\cdots-\confrac{1}{B_n(x)-1-x},
\end{equation*}
where $B_i(x)=\lfloor 1/(1-T_1^{i-1}x)\rfloor+1$ for $i=1,\ldots,n$. Since this fraction can be represented by $ax+b/(cx+d)$ for some $a,b,c,d\in\mathbb Z$ with $ad-bc=1$
(see e.g., \cite{IosKra02}), $x$ is a quadratic irrational.
As in the first paragraph, $(1-x)^{-1}$ is a quadratic irrational whose Galois conjugate is $(1-x^\dagger)^{-1}$.
Since the backward continued fraction expansion of $x$ is periodic, 
the minus continued fraction expansion of $(1-x)^{-1}$ is periodic. Proposition~\ref{katok-lem} yields $0<(1-x^\dagger)^{-1}<1$, and so $x^\dagger<0$ as required.
\end{proof}

\subsection*{Acknowledgments} 
We thank Karma Dajani and Cor Kraaikamp for fruitful discussions during their visit to Keio University. SS was supported by the JSPS KAKENHI 24K16932, Grant-in-Aid for Early-Career Scientists. 
HT was supported by the JSPS KAKENHI 
 25K21999, Grant-in-Aid for Challenging Research (Exploratory).

\subsection*{
 Data Availability} This article has no associated data and material.

\subsection*{Conflict of interest}
The authors have no conflicts of interest to declare that are relevant to the content of this article.

\end{document}